\documentclass[10pt]{amsart}

\usepackage{setspace} 
\usepackage{fullpage} 
\usepackage{lmodern} 
\usepackage{graphicx} 
\usepackage{color} 
\usepackage{float} 
\usepackage{accents} 
\usepackage{enumitem} 
\numberwithin{equation}{section}

\usepackage{comment}

\usepackage[numbers]{natbib}

\usepackage{amsmath, amsfonts, amssymb, amsthm} 
\usepackage{mathtools} 
\usepackage{xfrac} 
\usepackage{faktor} 
\usepackage{bbm} 
\usepackage{mathrsfs} 
\usepackage[mathscr]{euscript} 
\usepackage{bm}
\usepackage{dsfont}

\newcommand{\R}{\mathbb{R}} 
\newcommand{\C}{\mathbb{C}} 

\newcommand{\N}{\mathbb{N}} 
\newcommand{\B}{\mathscr{B}} 
\newcommand{\Borel}{\mathsf{B}} 
\newcommand{\Lin}{\mathscr{L}} 
\newcommand{\F}{\mathtt{F}} 
\newcommand{\E}{\mathtt{E}} 
\newcommand{\Fo}{\mathcal{F}} 
\newcommand{\Leb}{\mathrm{L}} 
\newcommand{\Hs}{\mathcal{H}} 
\newcommand{\K}{{\widetilde{\mathcal{H}}}} 
\newcommand{\D}{\mathcal{D}} 
\newcommand{\A}{\mathbf{A}} 
\DeclareMathOperator{\tr}{tr} 

\newcommand{\mc}[1]{\mathcal{#1}}
\renewcommand{\sf}[1]{\mathsf{#1}}

\renewcommand{\rm}[1]{\mathrm{#1}}
\renewcommand{\tilde}[1]{\widetilde{#1}}
\renewcommand{\epsilon}{\varepsilon}
\renewcommand{\geq}{\geqslant}
\renewcommand{\leq}{\leqslant}

\newcommand{\ran}[1]{\mathrm{Ran}\hspace{-1.5pt}\left(#1\right)}
\renewcommand{\ker}[1]{\mathrm{Ker}\hspace{-1.5pt}\left(#1\right)}
\newcommand{\hra}{\hookrightarrow}

\newcommand{\inter}[1]{\overset{#1}{\hra}}

\newcommand{\norm}[1]{\lVert #1 \rVert}
\newcommand{\inn}[2]{\langle #1, #2 \rangle}

\newcommand{\Id}{I}
\renewcommand{\Re}{\mathop{\rm {Re}}\nolimits}

\newcommand{\bpsi}{\beta}
\newcommand{\m}{\mathfrak{m}}

\makeatletter
\newcommand{\opnorm}{\@ifstar\@opnorms\@opnorm}
\newcommand{\@opnorms}[1]{%
  \left|\mkern-1.5mu\left|\mkern-1.5mu\left|
   #1
  \right|\mkern-1.5mu\right|\mkern-1.5mu\right|
}
\newcommand{\@opnorm}[2][]{%
  \mathopen{#1|\mkern-1.5mu#1|\mkern-1.5mu#1|}
  #2
  \mathclose{#1|\mkern-1.5mu#1|\mkern-1.5mu#1|}
}
\makeatother

\theoremstyle{plain}
\newtheorem{proposition}{Proposition}[section]
\newtheorem{lemma}{Lemma}[section]
\newtheorem{corollary}{Corollary}[section]
\newtheorem{theorem}{Theorem}[section]

\theoremstyle{definition}
\newtheorem{definition}{Definition}[section]

\theoremstyle{remark}

\usepackage{hyperref}
\usepackage{cleveref}

\setlist[enumerate,1]{label=(\arabic*),ref=(\arabic*)}
\setlist[enumerate,2]{label=(\alph*),ref=(\arabic{enumi})(\alph*)}
\setlist[enumerate,3]{label=(\roman*),ref=(\arabic{enumi})(\alph{enumii})(\roman*)}
\setlist[enumerate,4]{label=(\Alph*),ref=(\arabic{enumi}-\alph{enumii}-\roman{enumiii}-\Alph*)}

\begin{document}

\title{A spectral theoretical approach for hypocoercivity applied to some degenerate hypoelliptic, and non-local operators}

\author{P.~Patie}
\address{School of Operations Research and Information Engineering, Cornell University, Ithaca, NY 14853.}
\email{pp396@cornell.edu}

\author{A.~Vaidyanathan}
\address{Center for Applied Mathematics, Cornell University, Ithaca, NY, 14853.}
\email{av395@cornell.edu}

\date{ \today }

\thanks{This work was partially supported by NSF Grant DMS-1406599 and a CNRS grant. Both authors are grateful for the hospitality of the  Laboratoire de Math\'ematiques et de leurs applications de Pau, where this work was initiated.}

\begin{abstract}
The aim of this paper is to offer an original and comprehensive spectral theoretical approach to the study of convergence to equilibrium, and in particular of the hypocoercivity phenomenon, for contraction semigroups in Hilbert spaces. Our approach rests on a commutation relationship for linear operators known as intertwining, and we utilize this identity to transfer spectral information from a known, reference semigroup $\tilde{P} = (e^{-t\tilde{\A}})_{t \geq 0}$ to a target semigroup $P$ which is the object of study. This allows us to obtain conditions under which $P$ satisfies a hypocoercive estimate with exponential decay rate given by the spectral gap of $\tilde{\A}$. Along the way we also develop a functional calculus involving the non-self-adjoint resolution of identity  induced by the intertwining relations. We apply these results in a general Hilbert space setting to two cases: degenerate, hypoelliptic Ornstein-Uhlenbeck semigroups on $\R^d$, and non-local Jacobi semigroups on $[0,1]^d$, which have been introduced and studied for $d=1$ in \cite{cheridito:2019}. In both cases we obtain hypocoercive estimates and are able to explicitly identify the hypocoercive constants.
\end{abstract}
\maketitle

\section{Introduction}
When a system has a steady-state, one is naturally interested in how quickly the dynamics convergence to this equilibrium. We think of such a system as being described by a contraction semigroup $P = (P_t)_{t \geq 0} = (e^{-t\A})_{t \geq 0}$ acting on a Hilbert space $\Hs$, and the equilibrium consisting of $P$-invariant vectors given by $\{f \in \Hs; \: P_t f = f, \: \forall t \geq 0\}$ with corresponding projection $P_\infty$. Of particular interest is an estimate of the form
\begin{equation*}
\norm{P_t f - P_\infty f}_\Hs \leq Ce^{-\gamma t} \norm{f - P_\infty f}_\Hs,
\end{equation*}
where $C \geq 1$ and $\gamma > 0$ are constants, which is said to be hypocoercive. The literature on this topic is very rich and active, and several elegant techniques have been developed; we mention, without aiming to be exhaustive, generalizations of the $\Gamma$-calculus by Baudoin \cite{baudoin:2017} and Monmarch\'e~\cite{monmarche:2019}, entropy functional techniques by Dolbeault et al.~\cite{dolbeault:2015, dolbeault:2018} and Arnold \cite{arnold:2014}, the shrinkage/enlargenment approach by Gualdini et al.~\cite{gualdani:2017}, Bouin et al.~\cite{bouin:2017} and Mischler and Mouhot~\cite{mischler:2016}, generalized quadratic and Dirichlet form approaches by Ottobre et al.~\cite{ottobre:2015} and Grothaus and Stilgenbauer \cite{grothaus:2016}, respectively, a weak Poincar\'e inequality approach by Grothaus and Wang \cite{grothaus:2017}, a direct spectral approach for some toy models by Gadat and Miclo \cite{gadat:2013}, and a spectral approach combined with techniques from non-harmonic analysis by Patie and Savov \cite{patie:2019c} and Patie et al.~\cite{patie:2019}. We also mention the fundamental memoir by Villani \cite{villani:2009}, noting that the techniques developed therein were inspired by the work of Talay \cite{talay:2002}. Now, when $C = 1$ and $P$ is self-adjoint in $\Hs$, the constant $\gamma$ can be identified as the spectral gap of the operator $\A$ and thus there is a clear connection with the spectral theory of the underlying generator; however, outside of this situation a description of the constants $C$ and $\gamma$ is, first, difficult to obtain, and, second, is often not connected to the spectrum of $\A$. The aim of this work is to offer a new and spectral approach to the hypocoercivity phenomenon. Our approach rests on investigating the commutation relationship, known as intertwining, given, for any $t\geq 0$, by
\begin{equation*}
P_t \Lambda = \Lambda \tilde{P}_t,
\end{equation*}
where $\Lambda:\K \to \Hs$ is a bounded, linear operator and $\tilde{P} = (e^{-t\tilde{\A}})_{t \geq 0}$ is a reference contraction semigroup on another Hilbert space $\K$. Our main results in this context assume that $\tilde{\A}$ is a normal operator with a spectral gap $\bm{\gamma}_1$, and we are able to show, under some conditions, that $P$ satisfies a hypocoercive estimate with exponential rate $\bm{\gamma}_1$, the spectral gap of the reference operator $\tilde{\A}$. As applications of these results we obtain hypocoercive estimates for degenerate, hypoelliptic Ornstein-Uhlenbeck semigroups on $\R^d$, and for non-local Jacobi semigroups on $[0,1]^d$, recently introduced and studied for $d=1$ in \cite{cheridito:2019}. In both cases we make explicit the two hypocoercive constants in terms of the initial data.

This paper is organized as follows. In the remainder of this section we consider a motivating example and some preliminaries. In \Cref{sec:HC-main-results} we state our main results in a general Hilbert space setting and in \Cref{sec:HC-applications} we present our application of these general results to degenerate, hypoelliptic Ornstein-Uhlenbeck semigroups and non-local Jacobi semigroups. Finally, in \Cref{sec:HC-proofs} we provide the proofs.



\subsection{A motivating example}


We present a motivating example from \cite{patie:2019c}, which served as an inspiration for this work. Denote by $P = (e^{-t\mathbf{G}})_{t \geq 0}$ and $\tilde{P} = (e^{-t\tilde{\mathbf{G}}})_{t \geq 0}$ the generalized and classical Laguerre semigroup, which are contraction semigroups on the spaces $\Leb^2(\nu)$ and $\Leb^2(\epsilon)$, respectively, where $\epsilon(x) = e^{-x}$, $x > 0$, and $\nu$ is the unique invariant probability density on $(0,\infty)$ for $P$, see \cite[Theorem 1.6(2)]{patie:2019c}. The operator $-\mathbf{G}$ acts on suitable $f$ via
\begin{equation*}
-\mathbf{G}f(x) = a^2 xf''(x) + (\mathsf{k} +a^2-x)f'(x) + \int_0^\infty \left(f(e^{-y}x)-f(y)+yxf'(x)\right) \Pi(x,dy),
\end{equation*}
where, in what follows, we consider $a^2 > 0$, $\mathsf{k} \geq 0$ and $\Pi(x,dy) = x^{-1}\Pi(dy)$ with $\Pi$ a finite non-negative Radon measure on $(0,\infty)$ satisfying $\int_0^\infty (y^2 \wedge y) \Pi(dy) < \infty$. Note that $-\tilde{\mathbf{G}}$ is given from the above formula by setting $a^2 = 1$, $\sf{k}=0$, and $\Pi \equiv 0$, and that in \cite{patie:2019c} the authors treat a much wider class of parameters. For each generalized Laguerre semigroup $P$, there exists a bounded linear operator $\Lambda:\Leb^2(\epsilon) \to \Leb^2(\nu)$ with dense range such that, for all $t \geq 0$ and on $\Leb^2(\epsilon)$,
\begin{equation}
\label{eq:gen-Laguerre-intertwining}
P_t \Lambda = \Lambda \tilde{P}_t.
\end{equation}
Recall that $\tilde{P}$, as a self-adjoint and compact semigroup on $\Leb^2(\epsilon)$, is diagonalized by an orthonormal basis $(\mathcal{L}_n)_{n \geq 0}$ of $\Leb^2(\epsilon)$ formed of Laguerre polynomials, i.e.~for $f \in \Leb^2(\epsilon)$ and $t \geq 0$, we have $\tilde{P}_t f = \sum_{n=0}^\infty e^{-nt} \inn{f}{\mc{L}_n}_{\Leb^2(\epsilon)} \mc{L}_n$. This fact, together with \eqref{eq:gen-Laguerre-intertwining}, gives, for $t \geq 0$ and on the dense subspace $\ran{\Lambda}$,
\begin{equation}
\label{eq:expansion-gen-Laguerre}
P_t f = \sum_{n=0}^\infty e^{-nt} \inn{\Lambda^\dagger f}{\mc{L}_n}_{\Leb^2(\nu)} \mc{P}_n,
\end{equation}
where $\Lambda^\dagger$ denotes the pseudo-inverse of $\Lambda$, and $\mc{P}_n = \Lambda \mc{L}_n$ is a Bessel sequence, i.e.~for $f \in \Leb^2(\nu)$ we have $\sum_{n=0}^\infty |\inn{f}{\mc{P}_n}_{\Leb^2(\nu)}|^2 \leq  \norm{f}_{\Leb^2(\nu)}^2$ ($\Lambda$ has operator norm 1). For this subclass of generalized Laguerre semigroups there exists $(\mc{V}_n)_{n \geq 0} \in \Leb^2(\nu)$ solving the equation $\Lambda^* \mc{V}_n = \mc{L}_n$, where $\Lambda^*$ denotes the Hilbertian adjoint of $\Lambda$, see~Section 8 of the aforementioned paper. It follows that $(\mc{P}_n)_{n \geq 0}$ and $(\mc{V}_n)_{n \geq 0}$ are biorthogonal, i.e.~$\inn{\mc{P}_n}{\mc{V}_k}_{\Leb^2(\nu)} = 1$ if $n=k$ and 0 otherwise, but as $(\mc{V}_n)_{n \geq 0}$ is not itself a Bessel sequence we cannot substitute $\Lambda^* \mc{V}_n$ for $\mc{L}_n$ in \eqref{eq:expansion-gen-Laguerre}. Nevertheless, the multiplier sequence given by $m_n^2 = \frac{\Gamma(n+1)\Gamma(\mathsf{m}+1)}{\Gamma(n+\mathsf{m}+1)}$, where $\mathsf{m} = a^{-2}\left(\mathsf{k}+\int_0^\infty y\Pi(dy)\right) < \infty$, is such that $(m_n \mc{V}_n)_{n \geq 0}$ becomes a Bessel sequence, and consequently $\Leb^2(\nu) \ni f \mapsto \sum_{n=0}^\infty \inn{f}{m_n \mc{V}_n}_{\Leb^2(\nu)} \mc{P}_n$ defines a bounded linear operator. Furthermore, there exists a constant $T_m > 0$ such that, for $t > T_m$, $\sup_{n \geq 1} (m_n e^{nt})^{-1} \leq \sqrt{\mathsf{m}+1}e^{-t}$  and for any $f \in \Leb^2(\nu)$,
\begin{equation*}
P_t f = \sum_{n=0}^\infty e^{-nt} \inn{f}{\mc{V}_n}_{\Leb^2(\nu)} \mc{P}_n.
\end{equation*}
A consequence of the above spectral expansion for $P$ is the hypocoercive estimate
\begin{equation*}
\left\lVert P_t f - \int_0^\infty f(x)\nu(x)dx \right\rVert_{\Leb^2(\nu)} \leq \sqrt{\mathsf{m}+1}e^{-t} \left\lVert f - \int_0^\infty f(x)\nu(x)dx \right\rVert_{\Leb^2(\nu)},
\end{equation*}
which holds for all $t > 0$ and any $f \in \Leb^2(\nu)$, noting that, as the only $P$-invariant functions are constant, $P_\infty f = \int_0^\infty f(x)\nu(x)dx$. In this paper we provide a comprehensive framework that generalizes this approach, wherein the reference semigroup admits merely a spectral gap and does not necessarily have a discrete point spectrum, neither is necessarily compact.

\subsection{Preliminaries}

For a (real or complex) separable Hilbert space $\Hs$ we write $\inn{\cdot}{\cdot}_\Hs$ and $\norm{\cdot}_\Hs$ for the inner product and norm, respectively. Given two Hilbert spaces $\Hs$ and $\K$ we write $\B(\Hs,\K)$ for the space of bounded linear operators from $\Hs$ to $\K$, with norm $\norm{\cdot}_{\Hs \to \K}$, writing simply $\B(\Hs)$ for the unital Banach algebra of bounded linear operators on $\Hs$. Next, recall that a mapping $P:[0,\infty) \to \B(\Hs)$ is said to be a \emph{strongly continuous contraction semigroup}, or simply \emph{contraction semigroup} for short, if
\begin{enumerate}
\item $P_0 = \Id$, where $\Id$ is the identity on $\Hs$,
\item $P_{t + s} = P_t P_s$ for any $t, s \geq 0$,
\item $\norm{P_t}_{\Hs \to \Hs} \leq 1$ for all $t \geq 0$,
\item and $\lim_{t \to 0} \norm{P_t f - f}_\Hs$ for all $f \in \Hs$.
\end{enumerate}
For a contraction semigroup $P$ let
\begin{equation*}
\D(-\A) = \left\lbrace f \in \Hs; \ \lim_{t \to 0} \frac{P_t f - f}{t} \text{ exists} \right\rbrace, \quad \text{and} \quad -\A f = \lim_{t \to 0} \frac{P_t f - f}{t}, \enskip \forall f \in \Hs.
\end{equation*}
The operator $(-\A,\D(-\A))$ is generator of the semigroup $P$, which justifies writing $P = (e^{-t\A})_{t \geq 0}$, and we adopt this convention in order to have, by the Hille-Yosida Theorem, that the spectrum of $\A$ is contained in $\{z \in \C; \: \Re(z) \geq 0\}$. When $\A$ is a normal operator then $P = (e^{-t\A})_{t \geq 0}$ also holds in the sense of the Borel functional calculus for $\A$, see e.g.~\cite{rudin:1991,berezansky:1996a}. We refer to the excellent monographs \cite{davies:1980,engel:2000} for further aspects on the theory of one-parameter semigroups. Next, recall that $P_\infty$ denotes the orthogonal projection onto the closed subspace $\{f \in \Hs; \: P_t f = f, \: \forall t \geq 0\}$ of $P$-invariant vectors.

\begin{definition}
We say that $P$ \emph{converges to equilibrium with rate $r(t)$} if, for all $f \in \Hs$ and $t$ large enough,
\begin{equation}
\label{eq:convergence to equilibrium}
\norm{P_t f - P_\infty f}_\Hs \leq r(t) \norm{f - P_\infty f}_\Hs,
\end{equation}
where $\lim_{t\to\infty} r(t) = 0$. In the case when, for some $C \geq 1$ and $\gamma > 0$,
\begin{equation}
\label{eq:hypocoercive-estimate}
\norm{P_t f - P_\infty f}_\Hs \leq Ce^{- \gamma t} \norm{f- P_\infty f}_\Hs
\end{equation}
then we say that $P$ satisfies a \emph{hypocoercive} estimate.
\end{definition}

Note that our definition of hypocoercivity for a contraction semigroup $P = (e^{-t\A})_{t \geq 0}$ agrees with the definition (on the semigroup level) given by Villani in \cite[Chapter 3]{villani:2009}, when $\ran{P_\infty} = \ker{\A}$. However, for our purposes, it is useful to maintain a definition of convergence to equilibrium, and of hypocoercivity, purely on the semigroup level. When $P = (e^{-t\A})_{t \geq 0}$ is a normal semigroup and satisfies a hypocoercive estimate with $C=1$ and $\gamma > 0$ then $\gamma$ is a gap in the spectrum of $\A$, in which case \eqref{eq:hypocoercive-estimate} is also known as the spectral gap inequality see~\cite[Section 4.2]{bakry:2014}. Indeed, for the converse assertion, assuming that $P$ is normal and that $\A$ admits a spectral gap $\bm{\gamma}_1 > 0$, one gets that, for any $f \in \Hs$ with $P_\infty f = 0$ and $t \geq 0$,
\begin{equation*}
||P_t f||_{\Hs}^2 = \int_{\sigma(\A)} e^{-2\Re(\gamma) t} d \inn{\E_\gamma f}{f}_\Hs = \int_{\{\Re(\gamma) \geq \bm{\gamma}_1\}} e^{-2\Re(\gamma) t} d \inn{\E_\gamma f}{f}_\Hs \leq e^{-2\bm{\gamma}_1 t} \norm{f}_\Hs^2,
\end{equation*}
where $\E$ is the unique resolution of identity associated to $\A$, see the proof of \Cref{lem:spectral-gap-inequality} below where we recall this classical argument in more detail. We mention that Miclo in \cite{miclo:2015} gives a sufficient condition for a self-adjoint operator to admit a spectral gap. However, in general, the constants $C$ and $\gamma$ in \eqref{eq:hypocoercive-estimate} may have little to do with the spectrum of $\A$, and one of the purposes of our work is to elucidate their role. Finally, we now state our definition of intertwining.

\begin{definition}
Two contraction semigroups $P$ and $\tilde{P}$ on $\Hs$ and $\K$, respectively, are said to \emph{intertwine} if there exists $\Lambda \in \B(\K,\Hs)$ such that, for all $t \geq 0$ and on $\K$,
\begin{equation*}
P_t \Lambda = \Lambda \tilde{P}_t.
\end{equation*}
The operator $\Lambda$ is called the \emph{intertwining operator} and we use the shorthand $P \inter{\Lambda} \tilde{P}$.
\end{definition}

\section{Main Results} \label{sec:HC-main-results}





\subsection{The similarity case}

Throughout this section $P$ and $\tilde{P} = (e^{-t\tilde{\A}})_{t \geq 0}$ shall denote contraction semigroups on Hilbert spaces $\Hs$ and $\K$, respectively. We think of $P$ as the object of interest and $\tilde{P}$ as a reference semigroup so that the intertwining $P \inter{\Lambda} \tilde{P}$ allows us to transfer properties from the reference to the target. The relation $\hra$ between contraction semigroups on Hilbert spaces is trivially reflexive and transitive but is, in general, not an equivalence relation due to the lack of symmetry. There are several ways that one can symmetrize this relation, one that involves further assumptions on the intertwining operator and another that is more structural. First, if $P \inter{\Lambda} \tilde{P}$ and the intertwining operator $\Lambda$ is a bijection then it is straightforward that $\hra$ defines an equivalence relation among contraction semigroups on Hilbert spaces. We denote the equivalence class of $\tilde{P}$ by $\mc{S}(\tilde{P})$, which we call the \emph{similarity orbit of $\tilde{P}$}. Hence,
\begin{equation*}
P \in \mc{S}(\tilde{P}) \iff  \exists \Lambda \in \B(\K,\Hs) \text{ a bijection s.t. } P_t = \Lambda \tilde{P}_t \Lambda^{-1}, \: \forall t \geq 0.
\end{equation*}
For a bijective operator $\Lambda \in \B(\K,\Hs)$ we denote its condition number by $\kappa(\Lambda) = \norm{\Lambda}_{\K \to \Hs} \norm{\Lambda^{-1}}_{\Hs \to \K} \geq 1$. Next, we write $\sigma(\tilde{\A}) \subset \C$ for the spectrum of $\tilde{\A}$ and $\Borel(\C)$ for the Borel subsets of the complex plane. Recall that a densely defined operator $\tilde{\A}$ on $\K$ is normal if $\tilde{\A}{\tilde{\A}}^{\hspace{-0.05cm}*} = {\tilde{\A}}^{\hspace{-0.05cm}*}\tilde{\A}$, where $\tilde{\A}^{\hspace{-0.05cm}*}$ denotes its adjoint in $\K$. To every normal operator $\tilde{\A}$ on $\K$ there exists a unique (self-adjoint) resolution of identity $\E:\Borel(\C) \to \B(\K)$ such that, by the Borel functional calculus for $\tilde{\A}$,
\begin{equation*}
\tilde{P}_t = \int_{\sigma(\tilde{\A})} e^{-\gamma t} d\E_\gamma,
\end{equation*}
where we recall that, for each $\Omega \in \Borel(\C)$, $\E_\Omega$ is a self-adjoint projection and that, for $(f,g) \in \K \times \K$, $\gamma \mapsto d\inn{\E_\gamma f}{g}$ defines a complex valued measure on $\sigma(\tilde{\A})$, see e.g.~\cite{rudin:1991,berezansky:1996a}. Let $\Lin(\Hs)$ be the space of linear (not necessarily continuous) operators on $\Hs$ and write $D \subset_d \Hs$ if $D$ is a dense subset of $\Hs$. Then, we say that $\F:\Borel(\C) \to \Lin(\Hs)$ is a non-self-adjoint (nsa) resolution of identity if
\begin{enumerate}
\item there exists $D \subset_d \Hs$ such that for each $\Omega \in \Borel(\C)$, $\F_\Omega$ is a closed, linear operator with domain $D$,
\item for each $\Omega \in \Borel(\C)$, $\F_{\Omega} \neq \F_{\Omega}^*$,
\item $\F_{\emptyset} = 0$, $\F_{\C} = \Id$, and, for any subsets $\Omega_1,\Omega_2 \in \B(\C)$, $\F_{\Omega_1}\F_{\Omega_2} = \F_{\Omega_2}\F_{\Omega_1} = \F_{\Omega_1 \cap \Omega_2}$,
\item for a countable collection of pairwise disjoint subsets $(\Omega_i)_{i = 1}^\infty$ we have, in the strong operator topology,
\begin{equation*}
\F_{\cup_{i=1}^\infty \Omega_i} = \sum_{i=1}^\infty \F_{\Omega_i}.
\end{equation*}
\end{enumerate}
We shall always write $\F$ for a nsa resolution of identity, keeping the notation $\E$ exclusively for a self-adjoint resolution of identity, and this notion has been studied, with $\C$ replaced by $\R$, by Burnap and Zwiefel \cite{burnap:1986}. A semigroup $P$ is a spectral operator in the sense of Dunford \cite{dunford:1954,dunford:1958} if there is a uniformly bounded nsa resolution of identity $\F$ commuting with $P$, and is of scalar type if, for all $t \geq 0$,
\begin{equation*}
P_t = \int_{\sigma(\A)} e^{-\gamma t} d\F_\gamma.
\end{equation*}
We refer to \cite{dunford:1988b} for more on the theory of scalar and spectral operators. The following result, proved in \Cref{subsec:proof-prop:convergence-similarity}, highlights a first connection between intertwining and convergence to equilibrium.

%

\begin{proposition}
\label{prop:convergence-similarity}
Suppose that $P \in \mc{S}(\tilde{P})$. If $\tilde{P}$ converges to equilibrium with rate $r(t)$ then $P$ converges to equilibrium with rate $\kappa(\Lambda) r(t)$. In particular if $\tilde{P}$ satisfies a hypocoercive estimate with constants $C \geq 1$ and $\lambda > 0$, as in \eqref{eq:hypocoercive-estimate}, then $P$ satisfies a hypocoercive estimate with constants $C\kappa(\Lambda)$ and $\lambda$. Furthermore, if $\tilde{P}$ is a normal semigroup then $P$ is a scalar, spectral operator in the sense of Dunford.
\end{proposition}


The idea of classifying and studying contraction semigroups via their similarity orbit has been used, in the context of transition semigroups of Markov chains, in \cite{choi:2018a,choi:2018} where the authors study more than simply convergence to equilibrium. As a concrete example to which the above proposition applies, one can take $\tilde{P}$ to be a normal, contraction semigroup on $\R^d$, $d \geq 1$, and let $\Lambda f(x) = f(Vx)$, where $V$ is an invertible, $d$-dimensional matrix. Then the semigroup $P$ defined via $P_t = \Lambda \tilde{P}_t \Lambda^{-1}$, $t \geq 0$, is a scalar, spectral operator.


\subsection{Beyond the similarity case}

In this section we go beyond the case when $P$ is a scalar, spectral operator in sense of Dunford, and when the intertwining operator is a bijection. To this end we need the following notion.



%

\begin{definition}[Proper intertwining]
\label{def:proper-intertwining}
Let $P \inter{\Lambda} \tilde{P} = (e^{-t\tilde{\A}})_{t \geq 0}$, where $\tilde{\A}$ is a normal operator. We say that $\Lambda$ is a \emph{proper} intertwining operator if $\ran{\Lambda} \subset_d \Hs$ and if, for any $\Omega \in \Borel(\C)$,
\begin{equation*}
\E_{\Omega} \left(\overline{\ran{\Lambda^*}} \right) \subseteq \overline{\ran{\Lambda^*}},
\end{equation*}
where $\E:\Borel(\C) \to \B(\K)$ is the unique resolution of identity associated to $\tilde{\A}$, and $\overline{\ran{\Lambda^*}}$ denotes the closure of $\ran{\Lambda^*}$. In such case we say that $P$ intertwines with $\tilde{P}$ \emph{properly}, or $P \inter{\Lambda} \tilde{P}$ properly, for short.
\end{definition}

We note that the second property of the definition holds trivially, and independently of $\E$, when $\ker{\Lambda} = \{0\}$. An operator $\Lambda \in \B(\K,\Hs)$ with $\ker{\Lambda} = \{0\}$ and $\ran{\Lambda} \subset_d \Hs$ is said to be a \emph{quasi-affinity}, and two semigroups $P$ and $\tilde{P}$ are said to be \emph{quasi-similar} if $P \inter{\Lambda} \tilde{P} \inter{\widetilde{\Lambda}} P$, with $\Lambda$ and $\tilde{\Lambda}$ being quasi-affinities. The study of quasi-similarities of contraction operators on Hilbert spaces was initiated by Sz.~Nagy and Foias, see~\cite{sz.-nagy:2010}. This notion yields another symmetrization of the relation $\hra$, and the results presented below may be viewed as extending the quasi-similar framework. We also mention that Antoine and Trapani \cite{antoine:2014} have studied quasi-similarity applied to pseudo-Hermitian quantum mechanics. Given $\Lambda \in \B(\K,\Hs)$ we write $\Lambda^\dagger$ for its pseudo-inverse, which is well-defined as $\Lambda$ is a closed, densely-defined linear operator, see e.g.~\cite[Chapter 9]{ben-israel:2003}. As a stepping stone towards convergence to equilibrium we establish the following.

\begin{proposition}
\label{prop:induced-nsa}
Suppose that $P \inter{\Lambda} \tilde{P} = (e^{-t\tilde{\A}})_{t \geq 0}$ properly and that $\tilde{\A}$ is a normal operator with unique resolution of identity $\E:\Borel(\C) \to \B(\K)$. Then the intertwining induces a nsa resolution of identity $\F:\Borel(\C) \to \Lin(\Hs)$ with domain $\ran{\Lambda}$ via
\begin{equation*}
\F_\Omega = \Lambda \E_\Omega \Lambda^\dagger.
\end{equation*}
Furthermore, for each $(f,g) \in \ran{\Lambda} \times \Hs$, $\gamma \mapsto \inn{\F_\gamma f}{g}$ defines a complex-valued measure, and for all $t\geq 0$,
\begin{equation*}
P_t = \int_{\sigma(\tilde{\A})} e^{-\gamma t} d\F_\gamma
\end{equation*}
on $\ran{\Lambda}$, in the sense that
$\inn{P_t f}{g}_\Hs = \int_{\sigma(\tilde{\A})} e^{-\gamma t} d\inn{\F_\gamma f}{g}_\Hs$.
\end{proposition}


This result is proved in \Cref{subsec:proof-prop:induced-nsa}. Note that the intertwining $P \inter{\Lambda} \tilde{P}$ allows $P_t$ to be expressed as a spectral integral, with respect to the nsa resolution of identity induced by $\Lambda$, over the spectrum of $\tilde{\A}$, i.e.~
\begin{equation*}
e^{-t\A} = \int_{\sigma(\tilde{\A})} e^{-\gamma t} d\F_\gamma, \quad \text{on} \quad \ran{\Lambda}.
\end{equation*}
As we show in \Cref{lem:nsa-roi}, the function $\gamma \mapsto e^{-\gamma t}$ may be replaced more generally by any bounded measurable function on $\sigma(\tilde{\A})$ and thus we get a Borel functional calculus for $\A$, even though $\A$ itself is not necessarily normal. Let us mention that such a spectral integral with respect to an nsa resolution of identity has also been shown in Patie et al. \cite{patie:2019} in the context of Krein's spectral theory of strings, see~Corollary 2.6 therein.




Next, we say that a normal operator $\tilde{\A}$ on $\tilde{\Hs}$ with $\sigma(\tilde{\A}) \subseteq \{z \in \C; \: \Re(z) \geq 0\}$ has a \emph{spectral gap}, denoted by $\bm{\gamma}_1$, if
\begin{equation*}
\bm{\gamma}_1 = \inf\left\lbrace \Re(\gamma); \Re(\gamma) > 0, \: \gamma \in \sigma(\tilde{\A}) \right\rbrace = \inf\left\lbrace \frac{\Re \inn{\tilde{\A} f}{f}_\K}{\norm{f}_\K^2} ; 0 \neq f \in \D(\tilde{\A})\right\rbrace > 0.
\end{equation*}
We write $\Leb^\infty(\sigma(\tilde{\A}))$ for the space of complex-valued, bounded Borelian functions on $\sigma(\tilde{\A})$ equipped with the uniform norm $\norm{\cdot}_\infty$ and, for any complex valued measure $\mu$ we denote its total variation by $|\mu|$. The following is one of the main results of this work.



\begin{theorem}
\label{thm:general-convergence}
Let $P \inter{\Lambda} \tilde{P} = (e^{-t\tilde{\A}})_{t \geq 0}$ properly, and suppose that $\tilde{\A}$ is normal with spectral gap $\bm{\gamma}_1$. Assume that there exists a function $m \in \Leb^\infty(\sigma(\tilde{\A}))$ such that
\begin{enumerate}[label=(\alph*)]
\item \label{condition-a:thm:general-convergence} for $(f,g) \in \ran{\Lambda} \times \Hs$,
\begin{equation*}
\int_{\sigma(\tilde{\A})} |m(\gamma)| d|\inn{\F_\gamma f}{g}_\Hs| \leq \norm{f}_\Hs \norm{g}_\Hs,
\end{equation*}
where $\F$ is the nsa resolution of identity induced by the intertwining,
\item \label{condition-b:thm:general-convergence} and for $t > T_m > 0$, with $T_m$ a constant,
\begin{equation*}
\gamma \mapsto \frac{e^{ -\gamma t}}{m(\gamma)} \in \Leb^\infty(\sigma(\tilde{\A})).
\end{equation*}
\end{enumerate}
Then, we have the following.
\begin{enumerate}
\item \label{item-1:thm:general-convergence} For $t > T_m$, $\int_{\sigma(\tilde{\A})} e^{-\gamma t} d\F_\gamma$ extends to a bounded, linear operator on $\Hs$.
\item \label{item-2:thm:general-convergence} Let $M_t^{(\bm{\gamma}_1)} \in \Leb^\infty(\sigma(\tilde{\A}))$ be given by $M_t^{(\bm{\gamma}_1)}(\gamma) = \frac{e^{-\gamma t}}{m(\gamma)} \bm{1}_{\{\Re(\gamma) \geq {\bm{\gamma}_1}\}}$. Then, for all $f \in \Hs$ and $t > T_m$,
\begin{equation*}
\norm{P_t f - P_\infty f}_{\Hs} \leq \norm{M_t^{(\bm{\gamma}_1)}}_\infty \norm{f - P_\infty f}_\Hs.
\end{equation*}
If $M_t^{(\bm{\gamma}_1)}$ attains its supremum at ${\bm{\gamma}_1}$ then, for all $f \in \Hs$ and $t > T_m$,
\begin{equation*}
\norm{P_t f - P_\infty f}_\Hs \leq \frac{1}{|m(\bm{\gamma}_1)|} e^{- \bm{\gamma}_1 t} \norm{f - P_\infty f}_\Hs.
\end{equation*}
\end{enumerate}
\end{theorem}


This theorem is proved in \Cref{subsec:proof-thm:general-convergence} and in \Cref{thm:two-sided-intertwining} we provide a sufficient condition for \Cref{condition-a:thm:general-convergence} of the theorem to be fulfilled. Note that, except in the case when $\Lambda^{-1} \in \B(\Hs,\K)$, the function $m$ must be decreasing as $|\gamma| \to \infty$. Indeed, supposing that $|m(\gamma)| \geq c > 0$ for all $\gamma \in \sigma(\tilde{\A})$, the condition in \Cref{thm:general-convergence}\ref{condition-a:thm:general-convergence} yields
\begin{equation*}
c \int_{\sigma(\tilde{\A})} d|\inn{\F_\gamma f}{g}_\Hs|  \leq \norm{f}_\Hs \norm{g}_\Hs.
\end{equation*}
However, as we show in \Cref{lem:nsa-roi}, the measure $\gamma \mapsto \inn{\F_\gamma f}{g}_\Hs$ has total variation no greater than $\norm{\Lambda^\dagger f}_\Hs \norm{\Lambda}_{\K \to \Hs} \norm{g}_\Hs$ and thus, for a finite constant $K$, we deduce that $\norm{\Lambda^\dagger f}_\Hs \leq K \norm{f}_\Hs$. Similarly, the condition in \Cref{condition-b:thm:general-convergence} cannot hold for $t = 0$ except in the case when $\Lambda$ admits a bounded inverse. In this sense the function $m$ indicates the departure of $\F$ from being a uniformly bounded nsa resolution of identity, and the rate of convergence in \Cref{thm:general-convergence}\ref{item-2:thm:general-convergence} is given by the norm of an operator that measures this departure.

The second part of \Cref{thm:general-convergence}\ref{item-2:thm:general-convergence} provides a simple condition under which $P$ satisfies a hypocoercive estimate with a rate equal to the spectral gap of the normal operator $\tilde{\A}$ associated to the reference semigroup $\tilde{P}$. As mentioned earlier, this is not surprising given that intertwining transfers spectral information from the reference to the target semigroup. In \Cref{sec:HC-applications} below we will give examples of functions $m$ satisfying the conditions of \Cref{thm:general-convergence} and for which $M_t^{(\bm{\gamma}_1)}$ attains its supremum at the spectral gap $\bm{\gamma}_1$.
Finally, the fact that the small-time behavior for the rate of convergence may be different from exponential has been observed in the context of some toy models by Gadat and Miclo \cite{gadat:2013}, for degenerate, hypoelliptic Ornstein-Uhlenbeck semigroups by Monmarch\'e \cite{monmarche:2019}, see also \Cref{thm:OU-convergence} below, and in the context of some non-reversible Markov chains by Patie and Choi \cite{choi:2018,choi:2018a}. This suggest that studying hypocoercivity only for $t > T_m$ may be natural.



For the next result, we recall that a normal operator $\tilde{\A}$ is said to have simple spectrum if there exists a vector $v \in \K$ such that, for all non-negative integers $k,l$, $v \in \D(\tilde{\A}^{\hspace{-0.05cm}*^k}\tilde{\A}^l)$ and $\K$ is the closed linear span of $\{\tilde{\A}^{\hspace{-0.05cm}*^k}\tilde{\A}^l; \: k,l \geq 0\}$.

\begin{theorem}
\label{thm:two-sided-intertwining}
Let $P \inter{\Lambda} \tilde{P} = (e^{-t\tilde{\A}})_{t \geq 0} \inter{\tilde{\Lambda}} P$ properly, and suppose that $\tilde{\A}$ is normal with spectral gap $\bm{\gamma}_1$. If there exists $m \in \Leb^\infty(\sigma(\tilde{\A}))$ such that
\begin{equation*}
m(\tilde{\A}) = \widetilde{\Lambda} \Lambda,
\end{equation*}
then the condition in \Cref{thm:general-convergence}\ref{condition-a:thm:general-convergence} is fulfilled with the normalized function $m (\norm{\Lambda}_{\K \to \Hs} \norm{\widetilde{\Lambda}}_{\Hs \to \K})^{-1}$. In particular, if $\tilde{\A}$ has simple spectrum then there exists $m \in \Leb^\infty(\sigma(\tilde{\A}))$ such that $m(\tilde{\A}) = \widetilde{\Lambda} \Lambda$. If such a function $m$ also satisfies the condition in \Cref{condition-b:thm:general-convergence} then the conclusions of \Cref{thm:general-convergence} hold.
\end{theorem}

This theorem is proved in \Cref{subsec:proof-thm:two-sided-intertwining}. The observation that the composition of intertwining operators can equal a function of the generator has been made before, and has been used recently in \cite{miclo:2018,miclo:2019} and also \cite{cheridito:2019}. In particular, in \cite{miclo:2019} the authors introduce and study the notion of \emph{completely monotone intertwining relationships}, which corresponds to $m$ in \Cref{thm:two-sided-intertwining} being a completely monotone function, and obtain, among other things, entropic convergence and hypercontractivity in this manner.

We have the following corollary of \Cref{thm:two-sided-intertwining}, which follows at once from the observation that, if $P \inter{\Lambda} \tilde{P} \inter{\tilde{\Lambda}} P$ with $\Lambda$ and $\tilde{\Lambda}$ quasi-affinities, then $P^* \inter{\tilde{\Lambda}^*} \tilde{P}^* \inter{\Lambda^*} P^*$ with $\tilde{\Lambda}^*$ and $\Lambda^*$ being quasi-affinities.

\begin{corollary}
\label{cor:two-sided-intertwining}
Under the assumptions of \Cref{thm:two-sided-intertwining}, suppose that the intertwining operators $\Lambda$ and $\tilde{\Lambda}$ are quasi-affinities, and that the function $m \in \Leb^\infty(\sigma(\tilde{\A}))$ satisfies the condition in \Cref{thm:general-convergence}\ref{condition-b:thm:general-convergence}. Then the conclusions of \Cref{thm:general-convergence} hold upon replacing $P$ by its adjoint semigroup $P^* = (P_t^*)_{t \geq 0}$, and by replacing $\F$ by $\tilde{\F}$, the nsa resolution of identity induced by the intertwining $P^* \inter{\tilde{\Lambda}^*} \tilde{P}^*$.
\end{corollary}

This result gives that, under a mild strengthening of the hypothesis in \Cref{thm:two-sided-intertwining}, the adjoint semigroup may be also expressed as an integral over the spectrum of $\tilde{\A}^{\hspace{-0.05cm}*}$ with respect to another nsa resolution of identity.


\section{Applications} \label{sec:HC-applications}

\subsection{Hypoelliptic Ornstein-Uhlenbeck semigroups}

In this section we apply the results from the previous section to hypoelliptic Ornstein-Uhlenbeck semigroups on $\R^d, d\geq 1$. Without aiming to be exhaustive, we mention that \cite{lunardi:1997} and the series of papers \cite{metafune:2001,metafune:2002,metafune:2002a} have been important works on the Ornstein-Uhlenbeck semigroup, as well as the Ornstein-Uhlenbeck operator, and the main findings are collected nicely in \cite[Chapter 9]{lorenzi:2007}; the recent survey \cite{bogachev:2018}, which presents a thorough account on the state-of-the-art for Ornstein-Uhlenbeck semigroups, shows that these objects continue to be active areas of research.

Let $B$ be a matrix such that $\sigma(B) \subseteq \{z \in \C; \: \Re(z) > 0\}$ and suppose $Q$ is a positive semi-definite matrix such that, with
\begin{equation*}
Q_t = \int_0^t e^{-sB}Qe^{-sB^*}ds,
\end{equation*}
we have $\det Q_t > 0$, for all $t > 0$. In particular, this holds when $Q$ is invertible, which we call the non-degenerate case, although it can happen that $\det Q_t > 0$, for all $t > 0$, with $\det Q = 0$, which we call the degenerate case. Under these assumptions on $(Q,B)$, the hypoelliptic Ornstein-Uhlenbeck semigroup $P$ associated to $(Q,B)$ admits the representation
\begin{equation*}
P_t f(x) = \frac{1}{(2\pi)^{d/2} (\det Q_t)^{1/2}}\int_{\R^d} f(e^{-tB}x-y) e^{-\inn{Q_t^{-1} y}{y}/2} dy,
\end{equation*}
where $f$ is a bounded measurable function and $\inn{\cdot}{\cdot}$ denotes the Euclidean inner product in $\R^d$, and also extends to a contraction semigroup on the weighted Hilbert space $\Leb^2(\rho_\infty)$, which plays the role of $\Hs$ from the previous section, where
\begin{equation*}
\rho_\infty(x) = \frac{1}{(2\pi)^{d/2} (\det Q_\infty)^{1/2}} e^{-\inn{Q_\infty^{-1} x}{x}/2},
\end{equation*}
with $Q_\infty = \int_0^\infty e^{-tB}Qe^{-tB^*}ds$, and
\begin{equation*}
\Leb^2(\rho_\infty) = \left\lbrace f:\R^d \to \C \text{ measurable}; \: \norm{f}_{\Leb^2(\rho_\infty)}^2 = \int_{\R^d} |f(x)|^2 \rho_\infty(x)dx < \infty \right\rbrace.
\end{equation*}
In fact $\rho_\infty$ is the unique invariant measure of $P$ in the sense that, for any $f \in \Leb^2(\rho_\infty)$ and $t \geq 0$, $\int_{\R^d} P_t f(x) \rho_\infty(x) dx = \int_{\R^d} f(x) \rho_\infty(x)dx$, and, since the only $P$-invariant functions are constants, we get that the projection $P_\infty$ is given by $P_\infty f(x) = \int_{\R^d} f(x) \rho_\infty(x)dx$. The generator of the Ornstein-Uhlenbeck semigroup $P = (e^{-t\A})_{t \geq 0}$ acts on suitable functions $f$ via
\begin{equation*}
-\A f(x) = \frac{1}{2}\sum_{i,j=1}^d q_{ij} \partial_i\partial_j f(x) - \sum_{i,j=1}^d b_{ij}x_j \partial_j f(x) =  \frac{1}{2}\tr(Q \nabla^2)f(x) - \inn{Bx}{\nabla}f(x), \quad x \in \R^d,
\end{equation*}
and the condition $\det Q_t > 0$, for all $t > 0$, is equivalent to the hypoellipticity of $\frac{\partial}{\partial t} + \A$ in the $d+1$ variables $(t,x_1,\ldots,x_d$), hence the terminology. In \cite[Theorem 3.4]{metafune:2002} it was shown that the spectrum of $\A$ in $\Leb^2(\rho_\infty)$ is entirely determined by the matrix $B$, specifically that, writing $\N = \{0,1,2,\ldots\}$, $\sigma(\A) = \left\lbrace \sum_{i=1}^r k_i \lambda_i; \ k_i \in \N \right\rbrace$, where $\lambda_1,\ldots,\lambda_r$ are the distinct eigenvalues of $B$. Hence, in particular, the spectral gap $\bm{\gamma}_1$ of $\A$ is given by the smallest eigenvalue of $\frac{1}{2}(B+B^*)$. Recall that $\kappa(V)$ denotes the condition number of any invertible matrix $V$, and note that if $V$ is positive-definite then $\kappa(V) = v_{\mathrm{max}} / v_{\mathrm{min}}$, where $v_{\mathrm{max}}, v_{\mathrm{min}} > 0$ are the largest and smallest eigenvalues of $V$, respectively. The following is the main result of this section.

\begin{theorem}
\label{thm:OU-convergence}
Let $P$ be an Ornstein-Uhlenbeck semigroup associated to $(Q,B)$ such that $\ker{Q}$ does not contain any invariant subspace of $B^*$. Suppose that $B$ is diagonalizable with similarity matrix $V$, and that $\sigma(B) \subseteq (0,\infty)$. Then, there exists a non-degenerate, hypoelliptic Ornstein-Uhlenbeck semigroup $\tilde{P}$, self-adjoint on $\Leb^2(\tilde{\rho}_\infty)$, such that $P \inter{\Lambda} \tilde{P} \inter{\tilde{\Lambda}} P$, where $\Lambda$ and $\tilde{\Lambda}$ are quasi-affinities. Furthermore, setting $\bm{t} = \frac{1}{\bm{\gamma}_1} \log \kappa(V Q_\infty V^*)$, we have
\begin{equation*}
\tilde{\Lambda}\Lambda = \tilde{P}_{\bm{t}}.
\end{equation*}
Consequently, for any $f \in \Leb^2(\rho_\infty)$ and $t \geq 0$,
\begin{equation*}
\left\lVert P_t f - \int_{\R^d} f(x)\rho_\infty(x)dx\right\rVert_{\Leb^2(\rho_\infty)} \leq \kappa(V Q_\infty V^*) e^{- \bm{\gamma}_1 t} \left\lVert f- \int_{\R^d} f(x)\rho_\infty(x)dx\right\rVert_{\Leb^2(\rho_\infty)}.
\end{equation*}
\end{theorem}

This result is proved in \Cref{subsec:proof-thm:OU-convergence} and we proceed by offering some remarks. First, we emphasize that our result covers the case when $Q$ is degenerate, which has attracted a lot of research interest and seen several elegant techniques developed, see e.g.~\cite{gadat:2013,lelievre:2013,ottobre:2015,grothaus:2016,arnold:2018}. The difficulty in dealing with the degenerate case stems, in part, from the fact that a degenerate Ornstein-Uhlenbeck semigroup can never be normal on $\Leb^2(\rho_\infty)$, cf.~\cite[Lemma 3.3]{ottobre:2015}. We mention that Arnold and Erb \cite{arnold:2014} have already shown hypocoercivity, under our assumptions, with exponential rate given by the spectral gap $\bm{\gamma}_1$ and that Arnold et al.~\cite{arnold:2018} and Monmarch\'e \cite{monmarche:2019} have proved hypocoercivity with exponential rate $\bm{\gamma}_1$ without assuming that $B$ is diagonalizable. However, in contrast to these existing results, we are able to explicitly identify the constant in front of the exponential, i.e.~$\kappa(VQ_\infty V^*)$, in terms of the initial data $Q$ and $B$. In particular, if $B$ is symmetric then $V$ is unitary and $\kappa(VQ_\infty V^*) = \kappa(Q_\infty)$. Similar results have been obtained by Achleitner et al.~\cite{achleitner:2018}, and by Patie and Savov \cite{patie:2019c} in the context of generalized Laguerre semigroups, as well as Cheridito et al.~\cite{cheridito:2019} in the context of non-local Jacobi semigroups. Let us mention that the restriction $\sigma(B) \subseteq (0,\infty)$ was made only to simplify the computations involving the composition $\tilde{\Lambda}\Lambda$, and we believe that with some additional effort \Cref{thm:OU-convergence} holds for all diagonalizable matrices $B$ with $\sigma(B) \subseteq \{z \in \C; \: \Re(z) > 0\}$. Finally, we note that the intertwinings in \Cref{thm:OU-convergence} yield a completely monotone intertwining relationship, in the sense of \cite{miclo:2019}, between $P$ and $\tilde{P}$, and this stronger type of intertwining will be exploited to investigate, among other things, the hypercontractivity of Ornstein-Uhlenbeck semigroups.

\subsection{Non-local Jacobi semigroups}

In this section we consider the non-local Jacobi semigroup on $[0,1]^d$, whose generator is a non-local perturbation of the classical (local) Jacobi operator on $[0,1]^d$. Given $\bm{\gamma}_1>0$ and $\mu \in \R^d$ such that $\bm{\gamma}_1 > \mu_i > 0$, for all $i = 1,\ldots,d$, we recall that the classical Jacobi operator $-\tilde{\A}_\mu$ acts on smooth functions $f:[0,1]^d \to \R$ such that $f(x) = f_1(x_1)\cdots f_d(x_d)$ via
\begin{equation*}
\tilde{\A}_\mu f(x) = - \sum_{i=1}^d x_i(1-x_i) \partial_i^2 f(x) + \sum_{i=1}^d (\bm{\gamma}_1 x_i-\mu_i ) \partial_i f(x).
\end{equation*}
It generates the Jacobi semigroup $\tilde{P}^{(\mu)} = (e^{-t\tilde{\A}_\mu})_{t \geq 0}$, which is a self-adjoint contraction semigroup on the weighted Hilbert space $\Leb^2(\beta_\mu)$, where $\beta_\mu$ is the unique invariant measure of $\tilde{P}^{(\mu)}$ consisting of the product of beta densities on $[0,1]$. Moreover, $\bm{\gamma}_1$ is the spectral gap of $\tilde{\A}_\mu$ -- hence the notation -- and the spectral gap uniquely determines the spectrum of $\tilde{\A}_\mu$ in $\Leb^2(\beta_\mu)$, which is given by $\sigma(\tilde{\A}_\mu) = \{n(n-1)+\bm{\gamma}_1n; \: n \in \N\}$. We refer to~\cite[Section 2.7.4]{bakry:2014}, as well as \cite[Section 5]{cheridito:2019}, for a review of these objects.

The non-local Jacobi semigroup $P = (e^{-t\A})_{t \geq 0}$ on $[0,1]^d$ is the tensor product of the one-dimensional non-local Jacobi semigroups that have been recently introduced and studied in~\cite{cheridito:2019}. The generator $-\A$ acts on suitable product functions $f:[0,1]^d \to \R$, $f(x) = f_1(x_1) \cdots f_d(x_d)$, via
\begin{equation*}
\A f(x) = \tilde{\A}_\mu f(x) + \sum_{i=1}^d \int_0^1 f_i(y)h_i(x_i y^{-1})y^{-1}dy,
\end{equation*}
where we assume that, for each $i = 1,\ldots,d$, $h_i:(1,\infty) \to [0,\infty)$ is such that
\begin{equation}  \label{eq:cond}
-(e^yh_i(e^y))' \textrm{ is a finite non-negative Radon measure on $(0,\infty)$ with  $\bm{\gamma}_1 > \mu_i > 1+\int_1^\infty h_i(y)dy$.} \end{equation}
 Note that the condition  $\int_1^\infty h_i(y)dy < \infty$ is implied by the previous requirement.  This operator generates a contraction semigroup on the weighted Hilbert space $\Leb^2(\beta) = \Leb^2(\beta_1) \otimes \cdots \otimes \Leb^2(\beta_d)$, where each $\beta_i$ is the probability density on $[0,1]$ uniquely determined, for $n \geq 1$, by its moments
\begin{equation*}
\int_0^1 x^n \beta_i (x)dx = \prod_{k=1}^n \frac{\phi_i(k)}{k+\bm{\gamma}_1-1}, \quad \text{where} \quad \phi_i(u) = (\mu - 1) + u - \int_1^\infty y^{-u} h_i(y)dy,
\end{equation*}
and $\phi_i:[0,\infty)\to[0,\infty)$ is a Bernstein function, i.e.~$\phi \in  \rm{C}^\infty(\R_+)$, the space of infinitely differentiable functions on $\R^d$, with $\phi'$ a completely monotone function, see~\cite{schilling:2012}. The invariant measure $\beta(x)dx$ of $P$ is unique, and again we have that $P_\infty f = \int_{[0,1]^d} f(x) \bpsi(x)dx$, however, except in the trivial case $h \equiv 0$, the semigroup $P$ is non-self-adjoint on $\Leb^2(\beta)$. We refer to Section 2.1 of \cite{cheridito:2019} for detailed information, specifically Theorem 2.2 therein regarding the last two claims. In the following we use the notation $(a)_x = \Gamma(x+a)/\Gamma(a)$ for $a > 0$ and $x \geq 0$.

\begin{theorem}
\label{thm:Jacobi-convergence}
Let $P$ be a non-local Jacobi semigroup with parameters $\bm{\gamma}_1$, $\mu$ and $h_1,\ldots,h_d$ satisfying the conditions \eqref{eq:cond}. Then, for each $\m \in (\max \mu_i ,\bm{\gamma}_1)$ there exists a local Jacobi semigroup $\tilde{P}^{(\m)} = (e^{-\tilde{\A}_\m})_{t \geq 0}$ on $[0,1]^d$, with spectral gap $\bm{\gamma}_1$ and drift vector $(\m,\ldots,\m)$, such that $P \inter{\Lambda} \tilde{P}^{(\m)} \inter{\tilde{\Lambda}} P$, where $\Lambda$ and $\tilde{\Lambda}$ are quasi-affinities satisfying
\begin{equation*}
\tilde{\Lambda} \Lambda = F_\m(\tilde{\A}_\m), \quad \text{where} \quad \gamma \mapsto F_\m(\gamma) = \frac{(1)_\gamma}{(\m)_\gamma} \frac{(\bm{\gamma}_1-\m)_\gamma}{(\bm{\gamma}_1-1)_\gamma} \in \Leb^\infty(\sigma(\tilde{\A}_\m)).
\end{equation*}
Consequently, for any $\m \in (\max \mu_i ,\bm{\gamma}_1)$, $f \in \Leb^2(\beta)$ and $t \geq 0$,
\begin{equation*}
\left\lVert P_t f - \int_{[0,1]^d} f(x) \beta(x)dx  \right\rVert_{\Leb^2(\beta)} \leq \m \frac{(\bm{\gamma}_1-1)}{(\bm{\gamma}_1-\m)} e^{-\bm{\gamma}_1 t} \left\lVert f - \int_{[0,1]^d} f(x) \beta(x)dx  \right\rVert_{\Leb^2(\beta)}.
\end{equation*}
\end{theorem}

This result is proved in \Cref{subsec:proof-thm:Jacobi-convergence}. Note that the non-local components of $P$ may be different in each coordinate. Any homeomorphism $\mathrm{H} :[0,1]^d \to E \subset \R^d$ induces a non-local Jacobi semigroup $P^\mathrm{H}$ with state space $E$ and invariant measure $\beta^\mathrm{H}$, the image of $\beta$ under $\mathrm{H}$, and also a unitary operator $\Lambda_\mathrm{H} \in \B(\Leb^2(\beta),\Leb^2(\beta^\mathrm{H}))$ such that $P^\mathrm{H} \inter{\Lambda_\mathrm{H}} P$. Consequently, by a combination of \Cref{thm:Jacobi-convergence} and \Cref{prop:convergence-similarity} we deduce that $P^\mathrm{H}$ satisfies the same kind of hypocoercive estimate as $P$. In this way one may construct non-local dynamics on compact state spaces which are guaranteed to be hypocoercive. As a concrete example one may take $\mathrm{H}:[0,1]^d \to S^d$ to be the homeomorphism from $[0,1]^d$ to $S^d = \{x \in \R^d; \: x_1 + \cdots + x_d \leq 1, \: x_i \geq 0, \: i =1,\ldots,d \}$, the standard simplex in $d$-dimensions.

\section{Proofs} \label{sec:HC-proofs}

\subsection{Preliminaries} \label{subsec:preliminaries-proofs}

Before giving the proofs of the main theorems we state and prove some preliminary results. Recall that an idempotent $\Pi$ is any operator satisfying $\Pi^2 = \Pi$. The following simple result concerns the robustness of the convergence to equilibrium condition in \eqref{eq:convergence to equilibrium} when considering bounded idempotents different from $P_\infty$.

\begin{lemma}
\label{lem:convergence-idempotent}
Let $P$ be a contraction semigroup on a Hilbert space $\Hs$. If there exists an idempotent $\Pi \in \B(\Hs)$ such that, for all $f \in \Hs$ and $t$ large enough,
\begin{equation*}
\norm{P_t f - \Pi f}_\Hs \leq r(t) \norm{ f - \Pi f}_\Hs,
\end{equation*}
with $\lim_{t\to\infty} r(t) = 0$, then $\Pi = P_\infty$, and hence $P$ converges to equilibrium with rate $r(t)$.
\end{lemma}

\begin{proof}
Suppose that $f \in \ran{P_\infty}$. Then, by the convergence assumption and as $f$ is $P$-invariant, it follows that for $t$ large enough
\begin{equation*}
\norm{f - \Pi f}_\Hs = \norm{P_t f - \Pi f}_\Hs \leq r(t) \norm{f - \Pi f}_\Hs,
\end{equation*}
and hence $f = \Pi f$, i.e~$f \in \ran{\Pi}$. On the other hand if $f \in \ran{\Pi}$ then, for any $s \geq 0$ fixed and $t \geq 0$,
\begin{equation*}
\norm{P_s \Pi f - \Pi f}_\Hs \leq \norm{P_{s+t} f - \Pi f}_\Hs + \norm{P_{s+t}f - P_s\Pi f}_\Hs \leq \norm{P_{s+t} f - \Pi f}_\Hs + \norm{P_{t}f - \Pi f}_\Hs,
\end{equation*}
where the second inequality uses that $\norm{P_s}_{\Hs \to \Hs} \leq 1$. Taking the limit as $t \to \infty$ yields, by the convergence assumption, that $\Pi f = P_s \Pi f$, and thus $f \in \ran{P_\infty}$. To finish the proof we observe that for any $f \in \Hs$,
\begin{equation*}
\norm{\Pi f}_\Hs \leq \norm{P_t f}_\Hs + \norm{P_t f - \Pi f}_\Hs \leq \norm{f}_\Hs + \norm{P_t f - \Pi f}_\Hs,
\end{equation*}
and taking $t \to \infty$ yields $\norm{\Pi f}_\Hs \leq \norm{f}_\Hs$. This gives that $\norm{\Pi}_{\Hs \to \Hs} \leq 1$, however,  any idempotent satisfies $\norm{\Pi}_{\Hs \to \Hs} \geq 1$, and thus we deduce $\norm{\Pi}_{\Hs \to \Hs} = 1$. Consequently $\Pi$ must be an orthogonal projection, and since orthogonal projections are uniquely characterized by their range we get $\Pi = P_\infty$.
\end{proof}

\Cref{lem:convergence-idempotent} allows us to prove the norm convergence of $P$ to any bounded idempotent, a strategy we will use in the sequel. Using it we can establish the following classical result, which will also be used in the proofs below, and  we provide its proof for sake of completeness.

\begin{lemma}
\label{lem:spectral-gap-inequality}
Let $\tilde{P} = (e^{-t\tilde{\A}})_{t \geq 0}$ be a contraction semigroup on a Hilbert space $\K$ and suppose $\tilde{\A}$ is normal with spectral gap $\bm{\gamma}_1 > 0$. Let $\Omega$ be either $\{0\}$ or $i\R$. Then, for any $f \in \K$ and $t \geq 0$,
\begin{equation*}
\norm{\tilde{P}_t - \E_\Omega f}_\K \leq e^{-{\bm{\gamma}_1} t} \norm{f-\E_\Omega f}_\K,
\end{equation*}
where $\E:\Borel(\C) \to \B(\K)$ is the unique resolution of identity associated to $\tilde{\A}$. Consequently, $\E_{\{0\}} = \E_{i\R} = \tilde{P}_\infty$.
\end{lemma}

\begin{proof}
By the Borel functional calculus for $\A$ we have, for any $t\geq 0$ and $f \in \K$, writing $\overline{f} = f-\E_{\{0\}}f$,
\begin{align*}
\norm{\tilde{P}_t \overline{f}}_{\K}^2 = \norm{\E_{\{0\}} \overline{f}}_{\K}^2 + \int_{\sigma(\tilde{\A}) \setminus \{0\}} e^{-2\Re(\gamma)t} d\inn{\E_\gamma \overline{f}}{\overline{f}}_\K = \int_{\sigma(\tilde{\A}) \setminus \{0\}} e^{-2\Re(\gamma)t} d\inn{\E_\gamma \overline{f}}{\overline{f}}_\K \leq e^{-2{\bm{\gamma}_1} t} \norm{\overline{f}}_{\K}^2
\end{align*}
where the inequality uses the fact that $\tilde{\A}$ has spectral gap ${\bm{\gamma}_1}$. Next, by the spectral mapping theorem, see e.g.~\cite{rudin:1991}, we get that $\E_{\{0\}} = \tilde{P}_\infty$, and this may also be deduced from the Borel functional calculus for $\tilde{\A}$ via
\begin{equation*}
\norm{(\tilde{P}_t-\E_{\{0\}})\E_{\{0\}}f}_{\K}^2 = \int_{\sigma(\tilde{\A})} |e^{-\gamma t}-1|^2 d\inn{\E_\gamma \E_{\{0\}} f}{f}_\K = \int_{\{0\}} |e^{-\gamma t}-1|^2 d\inn{\E_\gamma f}{f}_\K.
\end{equation*}
Thus invoking \Cref{lem:convergence-idempotent} we conclude that $\tilde{P}$ satisfies the spectral gap inequality and, in particular, converges to equilibrium. Hence it remains to show that $\E_{\{0\}} = \E_{i\R}$. To this end, for any $f \in \K$ and $t \geq 0$, we have
\begin{equation*}
\norm{\tilde{P}_t f}_{\K}^2 = \int_{\sigma(\A)} e^{-2\Re(\gamma)t} d\inn{\E_\gamma f}{f}_\K = \norm{\E_{i\R} f}_{\K}^2 + \int_{\sigma(\A) \setminus i\R} e^{-2\Re(\gamma)t} d\inn{\E_\gamma f}{f}_\K.
\end{equation*}
Taking $f-\E_{\{0\}}f = f-\tilde{P}_\infty f$ in the above identity yields
\begin{equation*}
\norm{\tilde{P}_t f- \tilde{P}_\infty f}_{\K}^2 = \norm{\tilde{P}_t (f-\tilde{P}_\infty f)}_{\K}^2 = \norm{\E_{i\R} (f-\E_{\{0\}}f)}_{\K}^2 + \int_{\sigma(\A) \setminus i\R} e^{-2\Re(\gamma)t} d\inn{\E_\gamma (f-\E_{\{0\}}f)}{f-\E_{\{0\}}f}_\K.
\end{equation*}
The left-hand side converges to zero as $t \to \infty$ since $\tilde{P}$ converges to equilibrium, while the integral on the right-hand side is also easily seen to convergence to zero as $t \to \infty$. Hence, by orthogonality of $\E$, we get
\begin{equation*}
0 = \norm{\E_{i\R} (f-\E_{\{0\}}f)}_{\K}^2 = \norm{\E_{i\R} f-\E_{\{0\} \cap i\R}f}_{\K}^2 = \norm{\E_{i\R} f-\E_{\{0\}}f}_{\K}^2,
\end{equation*}
and since $f \in \K$ was arbitrary we get $\E_{\{0\}} = \E_{i\R}$ as desired.
\end{proof}

\subsection{Proof of \Cref{prop:convergence-similarity}} \label{subsec:proof-prop:convergence-similarity}

Since $\Lambda \in \B(\K,\Hs)$ is a bijection we get that its inverse satisfies $\Lambda^{-1} \in \B(\Hs,\K)$. Set
\begin{equation*}
\Pi = \Lambda \tilde{P}_\infty \Lambda^{-1},
\end{equation*}
where $\tilde{P}_\infty$ is the projection onto the set of $\tilde{P}$-invariant vectors. Then, as the composition of bounded operators we get that $\Pi \in \B(\Hs)$, and it is straightforward to check that $\Pi^2 = \Pi$, i.e.~$\Pi$ is a bounded idempotent. If $\tilde{P}$ converges to equilibrium with rate $r(t)$ then
\begin{equation*}
\norm{P_t f - \Pi f}_\Hs = \norm{\Lambda \tilde{P}_t \Lambda^{-1} f - \Lambda \tilde{P}_\infty \Lambda^{-1} f}_\Hs \leq \norm{\Lambda}_{\K \to \K} \norm{\tilde{P}_t \Lambda^{-1} f - \tilde{P}_\infty \Lambda ^{-1} f}_\Hs \leq \kappa(\Lambda) r(t) \norm{f - \tilde{P}_\infty f}_\Hs ,
\end{equation*}
and hence $P$ converges to equilibrium by \Cref{lem:convergence-idempotent}. The proof of the last claim is straightforward and hence omitted. \qed

\subsection{Proof of \Cref{prop:induced-nsa}} \label{subsec:proof-prop:induced-nsa}

To begin we establish the following lemma, and we say that $\Lambda \in \B(\K,\Hs)$ is proper with respect to a self-adjoint resolution of identity $\E$ if the conditions in \Cref{def:proper-intertwining} are fulfilled.

\begin{lemma}
\label{lem:nsa-roi}
Let $\tilde{\A}$ be a normal operator on $\K$ with unique self-adjoint resolution of identity $\E:\Borel(\C) \to \B(\K)$. Suppose $\Lambda: \K \to \Hs$ is a proper linear operator with respect to $\E$, and define $\F:\Borel(\C) \to \Lin(\Hs)$ via
\begin{equation*}
\F_\Omega = \Lambda \E_\Omega \Lambda^\dagger.
\end{equation*}
Then $\F:\Borel(\C) \to \Lin(\Hs)$ is a nsa resolution of identity with domain $\ran{\Lambda}$ and, for $(f,g) \in \ran{\Lambda} \times \Hs$, we have the following properties.
\begin{enumerate}
\item \label{item-1:lem:nsa-roi} The measure $\gamma \mapsto \inn{\F_\gamma f}{g}_\Hs$ is of bounded variation, and $\gamma \mapsto \inn{\F_\gamma f}{g}_\Hs = \inn{\E_\gamma \Lambda^\dagger f}{\Lambda^* g}_\K$.
\item \label{item-2:lem:nsa-roi} For each $m(\tilde{\A}) \in \Leb^\infty(\sigma(\tilde{\A}))$ there exists a unique closed, densely-defined linear operator
\begin{equation*}
\int_{\sigma(\tilde{\A})} m(\gamma) d \F_\gamma
\end{equation*}
with domain $\ran{\Lambda}$, which satisfies
\begin{equation*}
\inn{\Lambda m(\tilde{\A}) \Lambda^\dagger f}{g}_\Hs = \int_{\sigma(\tilde{\A})} m(\gamma) d\inn{\F_\gamma f}{g}_\Hs.
\end{equation*}
\item \label{item-3:lem:nsa-roi} For any $m \in \Leb^\infty(\sigma(\tilde{\A}))$,
\begin{equation*}
\int_{\sigma(\tilde{\A})} |m(\gamma)| d|\inn{\F_\gamma f}{g}_\Hs| = \int_{\sigma(\tilde{\A})} d\left|\inn{\F_\gamma \Lambda m(\tilde{\A}) \Lambda^\dagger f}{g}_\Hs \right|.
\end{equation*}
\item \label{item-4:lem:nsa-roi} For $m_1, m_2 \in \Leb^\infty(\sigma(\tilde{\A}))$ we have, on $\ran{\Lambda}$, the multiplicative property
\begin{equation*}
\left(\int_{\sigma(\tilde{\A})} m_1(\gamma) d \F_\gamma \right)\left(\int_{\sigma(\tilde{\A})} m_2(\gamma) d \F_\gamma \right)= \int_{\sigma(\tilde{\A})} (m_1m_2)(\gamma) d \F_\gamma.
\end{equation*}
\end{enumerate}
\end{lemma}

\begin{proof}
First we note that all the properties of the pseudo-inverse $\Lambda^\dagger$ used below are given in \cite[Theorem 9.2]{ben-israel:2003}, starting with the fact that $\D(\Lambda^\dagger) = \ran{\Lambda}$. Then, by assumption on $\ran{\Lambda}$, we get that, for each $\Omega \in \Borel(\C)$, $\F_\Omega$ is densely-defined, and the linearity of $\F_\Omega$ follows from the linearity of each of the factors in the definition. Since $\Lambda^\dagger$ is a closed operator, and both $\Lambda$ and $\E_\Omega$ are bounded, it follows that $\F_\Omega$ is closed. Next, let $\Omega_1,\Omega_2 \in \Borel(\C)$, so that
\begin{equation*}
\F_{\Omega_1}\F_{\Omega_2} = \Lambda \E_{\Omega_1} \Lambda^\dagger \Lambda \E_{\Omega_2} \Lambda^\dagger,
\end{equation*}
where we used the fact that $\F_\Omega(\ran{\Lambda}) \subseteq \ran{\Lambda}$. Since $\Lambda^\dagger \Lambda$ is the projection onto the closed subspace $\overline{\ran{\Lambda^*}}$, the assumption that $\Lambda$ is proper with respect to $\E$ then gives that
\begin{equation*}
\Lambda \E_{\Omega_1} \Lambda^\dagger \Lambda \E_{\Omega_2} \Lambda^\dagger =  \Lambda \E_{\Omega_1} \E_{\Omega_2} \Lambda^\dagger = \Lambda \E_{\Omega_1 \cap \Omega_2} \Lambda^\dagger = \F_{\Omega_1 \cap \Omega_2}.
\end{equation*}
Finally, we suppose that $(\Omega_i)_{i =1}^\infty \in \Borel(\C)$ is a countable collection of pairwise disjoint subsets. Then, by continuity of the inner product we get, for $(f,g) \in \ran{\Lambda} \times \Hs$,
\begin{equation*}
\inn{\F_{(\Omega_i)_{i =1}^\infty} f}{g}_\Hs = \sum_{i=1}^\infty \inn{\F_{\Omega_i} f}{g}_\Hs = \sum_{i=1}^\infty \inn{\E_{\Omega_i} \Lambda^\dagger f}{\Lambda^* g}_\K,
\end{equation*}
and the countable additivity of $\F$ follows from the same property for $\E$, which completes the proof that $\F:\Borel(\C) \to \Lin(\Hs)$ defines a nsa resolution of identity. As shown above, we have, for $\Omega \in \Borel(\C)$,
\begin{equation*}
\inn{\F_\Omega f}{g}_\Hs = \inn{ \Lambda \E_\Omega \Lambda^\dagger f}{g}_\Hs = \inn{\E_\Omega \Lambda^\dagger f}{\Lambda^* g}_\K,
\end{equation*}
and since $\gamma \mapsto \inn{\E_\gamma \Lambda^\dagger f}{\Lambda^* g}_\K$ has total variation $\norm{\Lambda^\dagger f}_\Hs \norm{\Lambda}_{\K \to \Hs} \norm{g}_\Hs$, we complete the proof of \Cref{item-1:lem:nsa-roi}. Next, let $s \in \Leb^\infty(\sigma(\tilde{\A}))$ be a simple function, i.e.~for $k \geq 1$,
\begin{equation*}
s(\gamma) = \sum_{i=1}^k \alpha_i \bm{1}_{\Omega_i}(\gamma),
\end{equation*}
where $\Omega_1,\ldots,\Omega_k \in \Borel(\C)$ are disjoint subsets and $\alpha_1,\ldots,\alpha_k \in \C$, so that by the Borel functional calculus for $\tilde{\A}$,
\begin{equation*}
s(\tilde{\A}) = \sum_{i=1}^k \alpha_i \E_{\Omega_i}.
\end{equation*}
Then, with $(f,g) \in \ran{\Lambda} \in \Hs$, we have that
\begin{equation*}
\inn{\Lambda s(\tilde{\A}) \Lambda^\dagger f}{g}_\Hs = \sum_{i=1}^k \alpha_i \inn{\Lambda \E_{\Omega_i} \Lambda^\dagger f}{g}_\Hs = \sum_{i=1}^k \alpha_i \inn{\F_{\Omega_i} f}{g}_\Hs = \int_{\sigma(\tilde{\A})} s(\gamma) d \inn{\F_\gamma f}{g}_\Hs.
\end{equation*}
Now let $\epsilon > 0$, $m \in \Leb^\infty(\sigma(\tilde{\A}))$ and choose a simple function $s$ such that $\norm{m - s}_{\infty} < \epsilon$. Then, using the above representation for $s$ we get that
\begin{align*}
\left| \inn{\Lambda m(\tilde{\A}) \Lambda^\dagger f}{g}_\Hs - \int_{\sigma(\tilde{\A})} m(\gamma) d\inn{\F_\gamma f}{g}_\Hs \right|
& \leq \left| \inn{\Lambda m(\A) \Lambda^\dagger f}{g}_\Hs - \inn{\Lambda s(\A) \Lambda^\dagger f}{g}_\Hs \right| \\
&\phantom{\leq} + \left|  \int_{\sigma(\tilde{\A})} (m-s)(\gamma)d\inn{\F_\gamma f}{g}_\Hs \right| \\
&\leq  \norm{\Lambda}_{\K \to \Hs} \norm{m - s}_{\infty} \norm{\Lambda^\dagger f}_\Hs \ \norm{g}_\Hs \\
&\phantom{\leq} + \int_{\sigma(\tilde{\A})} |m(\gamma)-s(\gamma)| d|\inn{\E_\gamma \Lambda^\dagger f}{\Lambda ^* g}_\K| \\
& \leq 2\epsilon \norm{\Lambda}_{\K \to \Hs} \norm{\Lambda^\dagger f}_\Hs \ \norm{g}_\Hs.
\end{align*}
Since $\epsilon$ was arbitrary it follows that
\begin{equation*}
\Lambda m(\tilde{\A}) \Lambda^\dagger = \int_{\sigma(\tilde{\A})} m(\gamma) d\F_\gamma
\end{equation*}
on $\ran{\Lambda}$, and the fact that $\Lambda m(\tilde{\A}) \Lambda^\dagger$ is closed follows immediately from the closedness of $\Lambda^\dagger$, which completes the proof of \Cref{item-2:lem:nsa-roi}. For the proof of \Cref{item-3:lem:nsa-roi} let again $s \in \Leb^\infty(\sigma(\tilde{\A}))$ be a simple function. Then, for $(f,g) \in \ran{\Lambda} \times \Hs$,
\begin{equation*}
\int_{\sigma(\tilde{\A})} |s(\gamma)| d|\inn{\F_\gamma f}{g}_\Hs| = \int_{\sigma(\A)} \sum_{i=1}^k |\alpha_i| \bm{1}_{\Omega_i}(\gamma) d|\inn{\F_\gamma f}{g}_\Hs| = \sum_{i=1}^k |\alpha_i| |\inn{\F_{\Omega_i} f}{g}_\Hs|,
\end{equation*}
while on the other hand, since the measure $\gamma \mapsto \inn{\F_\gamma \Lambda s(\tilde{\A}) \Lambda^\dagger f}{g}_\Hs$ is the sum of Dirac masses,
\begin{equation*}
\int_{\sigma(\tilde{\A})} d|\inn{\F_\gamma \Lambda s(\tilde{\A}) \Lambda^\dagger f}{g}_\Hs = \sum_{i=1}^k  |\inn{\alpha_i \F_{\Omega_i} f}{g}_\Hs | = \sum_{i=1}^k |\alpha_i| |\inn{\F_{\Omega_i} f}{g}_\Hs |.
\end{equation*}
For general $m \in \Leb^\infty(\sigma(\tilde{\A}))$ and given $\epsilon > 0$, let $s$ be a simple function such that $\norm{m-s}_\infty < \infty$. Write $\mu_m$ for the measure $\gamma \mapsto \inn{F_\gamma \Lambda m(\A) \Lambda^\dagger f}{g}_\Hs$, and similarly for $\mu_s$. Then, using what we just proved for simple functions, we get
\begin{align*}
\left| \int_{\sigma(\tilde{\A})} |m(\gamma)| d|\inn{\F_\gamma f}{g}_\Hs| - \int_{\sigma(\tilde{\A})} d|\inn{\F_\gamma \Lambda m(\A) \Lambda^\dagger f}{g}_\Hs | \right|
& \leq \int_{\sigma(\tilde{\A})} \left| |m(\gamma)|-|s(\gamma)| \right| d|\inn{\F_\gamma f}{g}_\Hs |  \\
&\phantom{\leq} + \left|  |\mu_m|(\sigma(\tilde{\A})) -  |\mu_s|(\sigma(\tilde{\A})) \right| \\
&\leq  \norm{\Lambda}_{\K \to \Hs} \norm{m-s}_{\infty} \norm{\Lambda^\dagger f}_\Hs \ \norm{g}_\Hs \\
&\phantom{\leq} + \int_{\sigma(\tilde{\A})}  d|\inn{\F_\gamma \Lambda \left(m(\tilde{\A})-s(\tilde{\A})\right)\Lambda^\dagger f}{ g}_\Hs| \\
& \leq 2\epsilon \norm{\Lambda}_{\K \to \Hs} \norm{\Lambda^\dagger f}_\Hs \ \norm{g}_\Hs,
\end{align*}
where in the second inequality we used the reverse triangle inequality for the sup-norm, while in the last inequality we used the reverse triangle inequality for the total variation norm together with linearity of the inner product in the first variable. This completes the proof of \Cref{item-3:lem:nsa-roi}. Finally, for the multiplicative property of the integrals we observe that, by the multiplicative property for $\E$, $\Lambda m_1(\tilde{\A}) \Lambda^\dagger \Lambda m_2(\tilde{\A}) \Lambda^\dagger = \Lambda m_1(\tilde{\A}) m_2(\tilde{\A})\Lambda^\dagger = \Lambda (m_1m_2)(\tilde{\A})\Lambda^\dagger$, where we again used that $\Lambda$ is proper with respect to $\E$.
\end{proof}

\begin{proof}[Proof of \Cref{prop:induced-nsa}]
Applying $\Lambda^\dagger$ to both sides of the intertwining $P \inter{\Lambda} \tilde{P}$ gives
\begin{equation*}
P_t \Lambda \Lambda^\dagger  = \Lambda \tilde{P}_t \Lambda^\dagger.
\end{equation*}
By \cite[Theorem 9.2(e)]{ben-israel:2003}, we have that $\Lambda\Lambda^\dagger$ is the projection onto $\overline{\ran{\Lambda}}$, and from $\ran{\Lambda} \subset_d \Hs$, we deduce that this projection is the identity on $\ran{\Lambda}$. Thus, together with \Cref{lem:nsa-roi}\ref{item-3:lem:nsa-roi} and the fact that $\tilde{P} = (e^{-t\tilde{\A}})_{t \geq 0}$ we conclude that, on $\ran{\Lambda}$,
\begin{equation*}
P_t = \Lambda \tilde{P}_t \Lambda^\dagger = \int_{\sigma(\tilde{\A})} e^{-\gamma t} d\F_\gamma,
\end{equation*}
and the remaining claims were proved in \Cref{lem:nsa-roi}.
\end{proof}

\subsection{Proof of \Cref{thm:general-convergence}} \label{subsec:proof-thm:general-convergence}

Let $M_t:\sigma(\A) \to \C$ be the function defined by
\begin{equation*}
M_t(\gamma) = \frac{e^{-\gamma t}}{m(\gamma)},
\end{equation*}
which, for $t > T_m$, belongs to $\Leb^\infty(\sigma(\tilde{\A}))$ by assumption. From the condition in \Cref{condition-a:thm:general-convergence} we deduce that, for $(f,g) \in \ran{\Lambda} \times \Hs$ and $t > T_m$,
\begin{equation*}
\left|\inn{\int_{\sigma(\tilde{\A})} e^{-\gamma t} d\F_\gamma f}{g}_\Hs\right| \leq \int_{\sigma(\tilde{\A})} |e^{-\gamma t}| d|\inn{\F_\gamma f}{g}_\Hs \leq \int_{\sigma(\tilde{\A})} \left|\frac{e^{-\gamma t}}{m(\gamma)}\right| |m(\gamma)| d|\inn{\F_\gamma f}{g}_\Hs| \leq \norm{M_t}_\infty \norm{f}_\Hs \norm{g}_\Hs.
\end{equation*}
Thus we conclude that, for $t > T_m$, the operator $\int_{\sigma(\tilde{\A})} e^{-\gamma t} d\F_\gamma$ is bounded on $\ran{\Lambda} \subset_d \Hs$, so that by invoking the bounded linear extension theorem we obtain a unique, continuous linear extension to all of $\Hs$. Next, let us write simply $\E_0$ and $\F_0$ in place of $\E_{\{0\}}$ and $\F_{\{0\}}$, respectively. Then, by evaluating the assumption in \Cref{condition-b:thm:general-convergence} at $\gamma = 0$ we get that the idempotent $\F_{0} = \Lambda \E_{0} \Lambda^{-1}$ satisfies, for all $(f,g) \in \ran{\Lambda} \times \Hs$,
\begin{equation*}
|\inn{\F_{0} f}{g}_\Hs| \leq \int_{\{0\}} \frac{1}{|m(\gamma)|} |m(\gamma)| d|\inn{\F_\gamma f}{g}_\Hs| \leq \frac{1}{|m(0)|} \int_{\sigma(\A)} |m(\gamma)| d|\inn{\F_\gamma f}{g}_\Hs| \leq \frac{1}{|m(0)|} \norm{f}_\Hs \norm{g}_\Hs,
\end{equation*}
and thus we deduce that $\F_0$ is bounded on $\ran{\Lambda}$. Since \Cref{lem:spectral-gap-inequality} gives that $\E_{0} = \tilde{P}_\infty$ we get
\begin{equation*}
P_t \F_{0} = \Lambda \tilde{P}_t \Lambda \Lambda^\dagger \E_{0} \Lambda^\dagger = \Lambda \tilde{P}_t \E_{0} \Lambda^\dagger = \F_{0},
\end{equation*}
so that $\F_{0}$ is invariant for $P$, and in particular
\begin{equation}
\label{eq:P_t-F_0-equality}
(P_t - \F_{0})(\Id - \F_{0}) = P_t - P_t\F_{0} - \F_{0} + \F_{0}^2 = P_t - \F_{0} - \F_{0} + \F_{0} = P_t - \F_{0},
\end{equation}
where both of these equalities hold on $\ran{\Lambda}$. Putting all of these observations together we get that, for $t > T_m$ and any $f \in \ran{\Lambda}$, 
\begin{align}
\norm{P_tf - \F_{0} f}_{\Hs}^2 &= \left| \int_{\sigma(\tilde{\A}) \setminus \{0\}} e^{-\gamma t} d\inn{\F_\gamma (f-\F_{0}f)}{P_tf - \F_{0} f}_{\Hs} \right| \leq \int_{\sigma(\tilde{\A}) \setminus \{0\}} |e^{-\gamma t}| d|\inn{\F_\gamma (f-\F_0f)}{P_tf - \F_0 f}_{\Hs}| \nonumber \\
&= \int_{\Re(\gamma) \geq {\bm{\gamma}_1}} |e^{-\gamma t}| d|\inn{\F_\gamma (f-\F_0f)}{P_tf - \F_0 f}_{\Hs}| \nonumber \\
&\leq\norm{M_t^{(\bm{\gamma}_1)}}_\infty \int_{\Re(\gamma) \geq {\bm{\gamma}_1}} |m(\gamma)| d|\inn{\F_\gamma (f-\F_0f)}{P_tf - \F_0 f}_{\Hs}| \nonumber \\
&\leq \norm{M_t^{(\bm{\gamma}_1)}}_\infty \norm{f - \F_0 f}_{\Hs} \norm{P_tf - \F_0 f}_{\Hs} \label{eq:P_t-estimate}
\end{align}
where, in order, we have used \eqref{eq:P_t-F_0-equality}, the representation for $P_t$ as a spectral integral, the fact that $\tilde{\A}$ admits a spectral gap ${\bm{\gamma}_1} > 0$ and finally the assumptions on the function $m$. Canceling $\norm{P_tf - \F_0 F}$ from both sides of the inequality in \eqref{eq:P_t-estimate} yields, for $t > T_m$ and $f \in \ran{\Lambda}$,
\begin{equation*}
\norm{P_t f - \F_0 f}_{\Hs} \leq \norm{M_t^{(\bm{\gamma}_1)}}_\infty \norm{f - \F_0 f}_{\Hs},
\end{equation*}
which extends by density, and the continuity of the involved operators, to all of $\Hs$. Then, invoking \Cref{lem:convergence-idempotent} completes the proof that $P$ converges to equilibrium, since plainly $\lim_{t \to \infty} \norm{M_t^{\bm{(\gamma}_1)}}_\infty = 0$. \qed

\subsection{Proof of \Cref{thm:two-sided-intertwining}} \label{subsec:proof-thm:two-sided-intertwining}

We shall provide two proofs of \Cref{thm:two-sided-intertwining}, one that invokes \Cref{thm:general-convergence} and hence is based on properties of nsa resolutions of the identity, and another that makes use of the Borel functional calculus for $\tilde{\A}$. We need a preliminary results regarding commuting operators. We say that an operator $M \in \B(\K)$ commutes with a closed, densely-defined operator $\tilde{\A}$ on $\K$ if for some $z \in \rho(\tilde{\A}) = \C \setminus \sigma(\tilde{\A})$, the resolvent set of $A$, we have
\begin{equation*}
MR_z(\tilde{\A}) = R_z(\tilde{\A}) M,
\end{equation*}
where $R_z(\tilde{\A})$ denotes the resolvent operator. The following two results are adapted from \cite{teschl:2014}, and we refer to \cite[Chapter 3]{teschl:2014} for the appropriate definitions.

\begin{lemma}
\label{lem:comm-op-simple-spectrum}
Let $\tilde{P} = (e^{-t\tilde{\A}})_{t \geq 0}$ be a contraction semigroup on $\K$ and suppose that $\tilde{\A}$ is normal and has simple spectrum. Then for operator $M \in B(\K)$
\begin{equation*}
M \tilde{P}_t = \tilde{P}_t M, \: \forall t \geq 0 \iff M = m(\tilde{\A}), \text{ for some } m \in \Leb^\infty(\sigma(\tilde{\A})).
\end{equation*}
\end{lemma}

\begin{proof}
First we will show that
\begin{equation}
\label{eq:MP-t}
M \tilde{P}_t = \tilde{P}_t M, \: \forall t \geq 0 \iff MR_z(\tilde{\A}) = R_z(\tilde{\A}) M.
\end{equation}
For the only if direction, if $\tilde{\A}$ is normal and commutes with a bounded operator $M$, then $M f(\tilde{A}) = f(\tilde{\A})M$ for any $f \in \Leb^\infty(\sigma(\tilde{\A}))$, which includes the exponential function $\gamma \mapsto e^{-\gamma t}$, for any $t \geq 0$. In the other direction we use the fact that, for $z \in \rho(\tilde{\A})$,
\begin{equation*}
R_z(\tilde{\A}) = \int_0^\infty e^{-zt} \tilde{P}_t dt.
\end{equation*}
Then, $M \tilde{P}_t = \tilde{P}_t M$, for all $t \geq 0$, implies that
\begin{equation*}
R_z(\tilde{\A})M = \int_0^\infty e^{-zt} \tilde{P}_t M dt = \int_0^\infty e^{-zt} M\tilde{P}_t dt = MR_z(\tilde{\A}).
\end{equation*}
Having established \eqref{eq:MP-t}, the claim follows by similar arguments as in the proof of~\cite[Theorem 4.8]{teschl:2014}.
\end{proof}

\begin{proof}[First proof of \Cref{thm:two-sided-intertwining}]
By assumption we have, for $\Lambda:\K \to \Hs$ and $\tilde{\Lambda}:\Hs \to \K$ proper intertwining operators,
\begin{equation*}
P_t \Lambda = \Lambda \tilde{P}_t \quad \text{and} \quad \tilde{\Lambda} P_t = \tilde{P}_t \tilde{\Lambda},
\end{equation*}
and by combining these two identities we deduce that $\tilde{P}_t \tilde{\Lambda} \Lambda = \tilde{\Lambda} \Lambda \tilde{P}_t$. Now, if $\tilde{\A}$ has simple spectrum then we may invoke \Cref{lem:comm-op-simple-spectrum} to get that there exists $m \in \Leb^\infty(\sigma(\tilde{\A}))$ such that
\begin{equation}
\label{eq:m}
m(\tilde{\A}) = \tilde{\Lambda} \Lambda.
\end{equation}
Next, by \Cref{lem:nsa-roi}\ref{item-3:lem:nsa-roi} and for $(f,g) \in \ran{\Lambda} \times \Hs$, we have
\begin{equation*}
\int_{\sigma(\tilde{\A})} |m(\gamma)| d|\inn{\F_\gamma f}{g}_{\Hs}| = \int_{\sigma(\tilde{\A})} d|\inn{\F_\gamma \Lambda m(\tilde{\A}) \Lambda^\dagger f}{g}_{\Hs}|.
\end{equation*}
Together with \eqref{eq:m} and \Cref{lem:nsa-roi}\ref{item-1:lem:nsa-roi}, this gives
\begin{equation*}
\int_{\sigma(\A)} |m(\gamma)| d|\inn{\F_\gamma f}{g}_{\Hs}| = \int_{\sigma(\tilde{\A})} d|\inn{\F_\gamma \Lambda \tilde{\Lambda} f}{g}_{\Hs}| \leq \norm{\Lambda}_{\K \to \Hs} \norm{\Lambda^\dagger \Lambda \tilde{\Lambda} f}_{\Hs} \norm{g}_{\Hs} \leq \norm{\Lambda}_{\K \to \Hs} \norm{\tilde{\Lambda}}_{\Hs \to \K} \norm{f}_{\Hs} \norm{g}_{\Hs}
\end{equation*}
where we also used that $\Lambda \Lambda^\dagger$ is the identity and that $\Lambda^\dagger \Lambda$ is a projection and thus a bounded operator with norm 1, see \cite[Theorem 9.2]{ben-israel:2003} for both of these claims. Invoking \Cref{thm:general-convergence} then completes the proof.
\end{proof}

\begin{proof}[Second proof of \Cref{thm:two-sided-intertwining}]
By assumption on the function $m$, there exists $T_m > 0$ such that for $t > T_m$,
\begin{equation}
\label{eq:def-phi-t}
\gamma \mapsto M_t(\gamma) = \frac{e^{-\gamma t}}{m(\gamma)} \in \Leb^\infty(\sigma(\tilde{\A}))
\end{equation}
and, by the Borel functional calculus for $\tilde{\A}$, it follows that $M_t(\tilde{\A})$ is a bounded operator for $t > T_m$. Next, by the intertwining $P \inter{\Lambda} \tilde{P}$ and by the multiplicative property of the Borel functional calculus for $\tilde{\A}$, we get, for $t > T_m$,
\begin{equation*}
\Lambda M_t(\tilde{\A}) \tilde{\Lambda} \Lambda = \Lambda M_t(\tilde{\A}) m(\tilde{\A}) = \Lambda e^{-t\tilde{\A}} = P_t \Lambda,
\end{equation*}
and thus we deduce that
\begin{equation*}
P_t = \Lambda M_t(\tilde{\A}) \tilde{\Lambda}
\end{equation*}
where the equality holds on $\ran{\Lambda}$. However, as the right-hand side is a bounded linear operator we get, by the bounded linear extension Theorem, that for $t >T_m$,
\begin{equation}
\label{eq:P-t-alternate-representation}
P_t = \Lambda M_t(\tilde{\A}) \tilde{\Lambda} \quad \text{on} \quad \Hs.
\end{equation}
Next, since $\norm{M_t}_\infty < \infty$ for $t > T_m$ and taking $\gamma = 0$ in the definition of $M_t$ in \eqref{eq:def-phi-t}, we get that
\begin{equation*}
\gamma \mapsto m_0(\gamma) = \frac{1}{m(\gamma)}\bm{1}_{\{\gamma = 0\}} \in \Leb^\infty(\sigma(\tilde{\A})),
\end{equation*}
and also note that $m_0(\tilde{\A})\E_{\{0\}} = \E_{\{0\}} m_0(\tilde{\A}) = m_0(\tilde{\A})$. Let $\Pi \in \B(\Hs)$ be the operator defined by $\Pi = \Lambda m_0(\tilde{\A}) \tilde{\Lambda}$. Using $m(\tilde{\A}) = \tilde{\Lambda}\Lambda$, the previous observation, and the multiplicative property of the Borel functional calculus for $\tilde{\A}$ we get
\begin{equation*}
\Pi^2 = \left(\Lambda m_0(\tilde{\A}) \tilde{\Lambda} \right)\left(\Lambda m_0(\tilde{\A}) \tilde{\Lambda} \right) = \Lambda m_0(\tilde{\A}) m(\tilde{\A}) m_0(\tilde{\A}) \tilde{\Lambda} =  \Lambda m_0(\tilde{\A}) \E_{\{0\}} \tilde{\Lambda} = \Lambda m_0(\tilde{\A}) \tilde{\Lambda} = \Pi,
\end{equation*}
and we conclude that $\Pi$ is a bounded idempotent. Now, by \eqref{eq:P-t-alternate-representation} we get, for $t > T_m$,
\begin{equation*}
P_t \Pi  = \left(\Lambda M_t(\tilde{\A})\tilde{\Lambda} \right)\left(\Lambda m_0(\tilde{\A}) \tilde{\Lambda} \right) = \Lambda M_t(\tilde{\A}) m(\tilde{\A}) m_0(\tilde{\A}) \tilde{\Lambda} = \Lambda M_t(\tilde{\A})\E_{\{0\}} \tilde{\Lambda} = \Lambda m_0(\tilde{\A}) \tilde{\Lambda} = \Pi f,
\end{equation*}
and thus it follows that
\begin{equation*}
(P_t - \Pi)(\Id - \Pi) = P_t - P_t \Pi - \Pi + \Pi^2 = P_t - \Pi - \Pi + \Pi = P_t - \Pi.
\end{equation*}
Recall that, for $t > T_m$, $M_t^{\bm{(\gamma}_1)}:\sigma(\tilde{\A}) \to \C$ is given by $M_t^{\bm{(\gamma}_1)}(\gamma) = \frac{e^{-\gamma t}}{m(\gamma)} \bm{1}_{\{\Re(\gamma) \geq {\bm{\gamma}_1}\}}$, and observe that
\begin{align*}
\norm{M_t - m_0}_\infty = \sup_{\gamma \in \sigma(\tilde{\A})} \left|\frac{e^{-\gamma t}}{m(\gamma)}\bm{1}_{\{\gamma\neq 0\}}\right| = \sup_{0 \neq \gamma \in \sigma(\tilde{\A})} \left|\frac{e^{-\gamma t}}{m(\gamma)}\right| = \sup_{\Re(\gamma) > 0} \left|\frac{e^{-\gamma t}}{m(\gamma)}\right| = \norm{M_t^{\bm{(\gamma}_1)}}_\infty
\end{align*}
where in the third equality we used the fact that $\E_{i\R} = \E_0$, which is a consequence of \Cref{lem:spectral-gap-inequality}, and for the last equality used that $\tilde{\A}$ admits a spectral gap ${\bm{\gamma}_1}$. Then, we conclude that, for any $f \in \Hs$ and $t > T_m$,
\begin{align*}
\norm{(P_t - \Pi)(f-\Pi f)}_\Hs &= \norm{\left(\Lambda M_t(\tilde{\A}) \tilde{\Lambda} - \Lambda m_0(\tilde{\A}) \tilde{\Lambda} \right)( f-\Pi f)}_{\Hs} = \norm{\Lambda\left(M_t(\tilde{\A}) - m_0(\tilde{\A})\right) \tilde{\Lambda} (f-\Pi f)}_{\Hs} \\
& \leq \norm{\Lambda}_{\K \to \Hs} \norm{\tilde{\Lambda}}_{\Hs \to \K} \norm{M_t-m_0}_\infty \norm{f-\Pi f}_{\Hs} \\
&= \norm{\Lambda}_{\K \to \Hs} \norm{\tilde{\Lambda}}_{\Hs \to \K} \norm{M_t^{\bm{(\gamma}_1)}}_\infty \norm{f-\Pi f}_{\Hs}.		
\end{align*}
\end{proof}

\subsection{Proof of \Cref{thm:OU-convergence}} \label{subsec:proof-thm:OU-convergence}

To give the proof of \Cref{thm:OU-convergence} we state and prove some auxiliary results that may be of independent interests. We write $\Fo_f$ for the Fourier transform of a suitable function $f$,  which for a function $f \in \Leb^1(\R^d)$ and any $\xi \in \R^d$, can be represented by
\begin{equation*}
\Fo_{f}(\xi) = \int_{\R^d} e^{i\inn{\xi}{x}} f(x)dx,
\end{equation*}
and use, when needed, the notation $e_{i\xi}:x \mapsto e^{i \inn{\xi}{x}}$. We also write $\prec$ for the L\"owner ordering of positive-definite matrices, that is, for two symmetric matrices $X$ and $Y$, $X \prec Y$ if and only if $X-Y$ is positive-definite. Hence $X \succ 0$ is shorthand for saying that $X$ is positive-definite.

\begin{proposition}
\label{prop:OU-intertwining}
Let $P$ and $\tilde{P}$ be two Ornstein-Uhlenbeck semigroups associated to $(Q,B)$ and $(\tilde{Q},B)$, respectively, and suppose that $0 \prec Q_\infty \prec \tilde{Q}_\infty$. Then the operator $\Lambda_{Q \to \tilde{Q}} = \Lambda : \Leb^2(\R^d) \to \Leb^2(\R^d)$ defined, for $f \in \Leb^2(\R^d)$, by
\begin{equation*}
\Fo_{\Lambda f} (\xi) = e^{-\inn{(\tilde{Q}_\infty - Q_\infty)\xi}{\xi}/2} \Fo_f(\xi), \quad \xi \in \R^d,
\end{equation*}
belongs to $\B(\Leb^2(\R^d))$. Moreover, $\Lambda \in \B(\Leb^2(\tilde{\rho}_\infty),\Leb^2(\rho_\infty))$ is a quasi-affinity with $\norm{\Lambda}_{\Leb^2(\tilde{\rho}_\infty) \to \Leb^2(\rho_\infty)} = 1$, and we have the intertwining $P \inter{\Lambda} \tilde{P}$.
\end{proposition}

For the proof of this proposition we shall need the following lemma, where we recall that any Ornstein-Uhlenbeck semigroup $P$ extends to a contraction semigroup from $\Leb^p(\R^d), p\geq 1,$ to itself, see e.g.~\cite[Proposition 9.4.1]{lorenzi:2007}.

\begin{lemma}
\label{lem:Fourier-OU}
Let $P$ be the Ornstein-Uhlenbeck semigroup associated to $(Q,B)$. Then, for any $f \in \Leb^2(\R^d)$ and $t \geq 0$,
\begin{equation*}
\Fo_{P_t f}(e^{-tB^*}\xi) = \frac{1}{|\det(e^{tB})|} e^{-\inn{Q_t \xi }{\xi}/2} \Fo_f(\xi), \quad \xi \in \R^d.
\end{equation*}
\end{lemma}

\begin{proof}
By a change of variables we have that
\begin{equation*}
P_tf(x) = \int_{\R^d} f(y) \rho_t(y-e^{-tB}x)dy.
\end{equation*}
Then, since for any $f \in \Leb^2(\R^d) \cap \Leb^1(\R^d)$,  $P_tf  \in \Leb^2(\R^d) \cap \Leb^1(\R^d)$, we get  using Fubini's Theorem
\begin{align*}
\Fo_{P_t f}(\xi) = \int_{\R^d} e^{i\inn{\xi}{x}} P_t f(x)dx &= \int_{\R^d}e^{i\inn{\xi}{x}} \int_{\R^d} f(y) \rho_t(y-e^{-tB}x)dy dx \\
&= \int_{\R^d} f(y) \int_{\R^d}  e^{i\inn{\xi}{x}} \rho_t(y-e^{-tB}x)dx dy \\
&=  \frac{1}{|\det(e^{tB})|} \int_{\R^d} f(y) \int_{\R^d}  e^{-i\inn{e^{tB^*}\xi}{x}} \rho_t(y+x)dx dy \\
&= \frac{1}{|\det(e^{tB})|} e^{-\inn{Q_t e^{tB^*}\xi}{e^{tB^*}\xi}/2} \int_{\R^d} e^{i\inn{e^{tB^*}\xi}{y}} f(y) dy,
\end{align*}
and the claim follows from $\Leb^2(\R^d) \cap \Leb^1(\R^d) \subset_d \Leb^2(\R^d)$ together with the continuity of the Fourier transform.
\end{proof}

\begin{proof}[Proof of \Cref{prop:OU-intertwining}]
First, we note that $\Lambda$ is a Fourier multiplier operator whose multiplier is, by the assumption  $\tilde{Q}_\infty - Q_\infty \succ 0$, a bounded measurable function; thus $\Lambda \in \B(\Leb^2(\R^d))$. By identifying the multiplier we deduce that $\Lambda$ is the convolution operator associated to the Gaussian measure with covariance matrix $(\tilde{Q}_\infty - Q_\infty)^{-1}$, i.e.~writing $\rho_\Lambda$ for this Gaussian measure, $\Lambda f(x) = f \ast \rho_\Lambda(x)$, with $\ast$ denoting the additive convolution operator. Clearly for any $f \in \Leb^2(\tilde{\rho}_\infty)$, $\Lambda f$ makes sense, and hence it remains to show the boundedness and that $\Lambda$ is a quasi-affinity. To this end we observe that the following factorization of measures holds
\begin{equation*}
\rho_\infty \ast \rho_\Lambda = \tilde{\rho}_\infty,
\end{equation*}
which follows from the identity,
\begin{equation*}
\Fo_{\rho_\infty}(\xi) \cdot \Fo_{\rho_\Lambda}(\xi) = e^{-\inn{Q_\infty \xi}{\xi}/2} e^{-\inn{(\tilde{Q}_\infty - Q_\infty)\xi}{\xi}/2} = e^{-\inn{\tilde{Q}_\infty \xi}{\xi}/2} = \Fo_{\tilde{\rho}_\infty}(\xi), \quad \xi \in \R^d,
\end{equation*}
together with the fact that the Fourier transform uniquely characterizes probability measures. Using this factorization we get, for any $f \in \Leb^2(\tilde{\rho}_\infty)$, and by appealing to Jensen's inequality
\begin{equation*}
\int_{\R^d} \left|\Lambda f(x)\right|^2 \rho_\infty(x)dx \leq \int_{\R^d} \Lambda |f|^2(x) \rho_\infty(x)dx = \int_{\R^d} |f(x)|^2 \tilde{\rho}_\infty(x)dx,
\end{equation*}
and thus $\Lambda \in \B(\Leb^2(\tilde{\rho}_\infty),\Leb^2(\rho_\infty))$ with $\norm{\Lambda}_{\Leb^2(\tilde{\rho}_\infty) \to\Leb^2(\rho_\infty)} \leq 1$; however, equality is achieved by the constant function $\bm{1} \in \Leb^2(\tilde{\rho}_\infty)$. Next we observe that, for any $\xi \in \R^d$, $\Lambda e_{i\xi}(x) = e^{-\inn{(\tilde{Q}_\infty - Q_\infty)\xi}{\xi}/2} e_{i\xi}(x)$. Thus $\ran{\Lambda}$ contains the linear span of the set $\{e_{i\xi}; \: \xi \in \R^d\}$ and to show that $\ran{\Lambda} \subset_d \Leb^2(\rho_\infty)$ it then suffices to show that this linear span is dense in $\Leb^2(\rho_\infty)$. To this end, we suppose there exists $g \in \Leb^2(\rho_\infty)$ such that $\inn{e_{i\xi}}{g}_{\Leb^2(\rho_\infty)} = 0$, for all $\xi \in \R^d$. Then, by standard application of Jensen's inequality, $g \in \Leb^2(\rho_\infty)$ implies that $g \in \Leb^1(\rho_\infty)$ and hence $(g\rho_\infty) \in \Leb^1(\R^d)$. Thus $(g\rho_\infty)$ is an integrable function with vanishing Fourier transform, and we conclude that $g = 0 \in \Leb^2(\rho_\infty)$. Now, for any $f \in \ker{\Lambda}$ we have that $f \ast \rho_\Lambda(x) =0$ a.e., which forces $f(x) = 0$ a.e.~and thus $f = 0$ in $\Leb^2(\tilde{\rho}_\infty)$, which gives that $\Lambda$ is a quasi-affinity. Finally, it remains to show that the intertwining relation $P \inter{\Lambda} \tilde{P}$ holds. Using the fact that $\Lambda \in \B(\Leb^2(\R^d))$, we have on the one hand, for $f \in \Leb^2(\R^d)$ and $\xi \in \R^d$,
\begin{align*}
\Fo_{P_t \Lambda f}(e^{-tB^*}\xi) = \frac{1}{|\det(e^{tB})|} e^{-\inn{Q_t \xi }{\xi}/2} \Fo_{\Lambda f}(\xi) &= \frac{1}{|\det(e^{tB})|} e^{-\inn{Q_t \xi }{\xi}/2} e^{-\inn{(\tilde{Q}_\infty - Q_\infty)\xi}{\xi}/2} \Fo_f(\xi) \\
&= \frac{1}{|\det(e^{tB})|} e^{ \inn{(Q_\infty -Q_t) \xi }{\xi}/2} e^{-\inn{\tilde{Q}_\infty\xi}{\xi}/2},
\end{align*}
while on the other hand,
\begin{equation*}
\Fo_{\Lambda \tilde{P}_t f}(e^{-tB^*}\xi) = \lambda(e^{-tB^*}\xi) \Fo_{\tilde{P}_t f}(e^{-tB^*}\xi) = \frac{1}{|\det(e^{tB})|}  e^{-\inn{(\tilde{Q}_\infty - Q_\infty)e^{-tB^*}\xi}{e^{-tB^*}\xi}/2}  e^{-\inn{\tilde{Q}_t \xi }{\xi}/2} \Fo_f(\xi),
\end{equation*}
where we used twice \Cref{lem:Fourier-OU}. Then, since
\begin{equation*}
e^{-tB}Q_\infty e^{-tB^*} = \int_0^\infty e^{-(t+s)B}Q e^{-(t+s)B^*} ds = \int_t^\infty e^{-sB}Q e^{-sB^*}ds = Q_\infty - Q_t
\end{equation*}
we get that
\begin{align*}
e^{-\inn{(\tilde{Q}_\infty - Q_\infty)e^{-tB^*}\xi}{e^{-tB^*}\xi}/2}  e^{-\inn{\tilde{Q}_t \xi }{\xi}/2} &= e^{- \inn{(\tilde{Q}_\infty - \tilde{Q}_t) \xi}{\xi}/2}  e^{\inn{(Q_\infty - Q_t)\xi}{\xi}/2 } e^{-\inn{\tilde{Q}_t \xi }{\xi}/2} \\
&= e^{\inn{(Q_\infty - Q_t)\xi}{\xi}/2 } e^{-\inn{\tilde{Q}_\infty\xi}{\xi}/2},
\end{align*}
and thus we conclude that, for any $f \in \Leb^2(\R^d)$ and $t \geq 0$,
\begin{equation*}
\Fo_{P_t \Lambda f}(e^{-tB^*}\xi) = \Fo_{\Lambda \tilde{P}_t f}(e^{-tB^*}\xi).
\end{equation*}
By the $\Leb^2$-isomorphism of the Fourier transform we then deduce that, for any $f \in \Leb^2(\R^d)$ and $t \geq 0$,
\begin{equation*}
P_t \Lambda f = \Lambda \tilde{P}_t f.
\end{equation*}
In particular, this holds for $f$ belonging to $\rm{C}_c^\infty(\R^d)$, the space of smooth, compactly supported functions on $\R^d$ and we have the inclusions $\rm{C}_c^\infty(\R^d) \subset_d \Leb^2(\rho_\infty)$ and $\rm{C}_c^\infty(\R^d) \subset_d \Leb^2(\tilde{\rho}_\infty)$. Hence, by density and the continuity of all involved operators, the claimed intertwining also holds on $\Leb^2(\tilde{\rho})$.
\end{proof}

In the following we write, for a vector $\alpha \in \R^d$, $D_\alpha$ for the diagonal matrix with diagonal entries given by $\alpha$. For $\alpha, \delta \in \R^d$ we denote by $D_{\alpha\delta} = \mathrm{diag}(\alpha_1\delta_1,\ldots,\alpha_d\delta_d) = D_\alpha D_\delta$.

\begin{proposition}
\label{prop:OU-quasi-affinity}
Let $P$ be an Ornstein-Uhlenbeck semigroup associated to $(Q,B)$, and suppose that $B=D_b$, with $b_i > 0$ for all $i$. For any $i \in \{1,\ldots,d\}$ set
\begin{equation*}
\alpha_i = q_{\infty,\mathrm{min}} \left(\frac{q_{\infty,\mathrm{max}}}{q_{\infty,\mathrm{min}}}\right)^{b_i / b_{\mathrm{min}}} \quad \text{and} \quad \delta_i = q_{\infty,\mathrm{min}},
\end{equation*}
where $q_{\infty,\mathrm{min}}$ and $q_{\infty,\mathrm{min}}$ are the largest and smallest eigenvalues of $Q_\infty$, respectively, and $b_\mathrm{min}$ is the smallest eigenvalue of $B$.
\begin{enumerate}
\item \label{item-1:prop:OU-quasi-affinity} We have $D_\alpha \succ Q_\infty \succ D_\delta$, and there exist matrices $Q^{(\alpha)}, Q^{(\delta)} \succ 0$ such that
\begin{equation*}
D_\alpha = \int_0^\infty e^{-sB} Q^{(\alpha)} e^{-sB} ds, \quad \text{and} \quad D_\delta = \int_0^\infty e^{-sB} Q^{(\delta)} e^{-sB} ds.
\end{equation*}
\item \label{item-2:prop:OU-quasi-affinity} The Ornstein-Uhlenbeck semigroups $P^{(\alpha)}$ and $P^{(\delta)}$ associated to $(Q^{(\alpha)},B)$ and $(Q^{(\delta)},B)$, respectively, are self-adjoint and $P \inter{\Lambda_\alpha} P^{(\alpha)}$, $P^{(\delta)} \inter{\Lambda_\delta} P$, and $P^{(\delta)} \inter{\Lambda_{\delta,\alpha}} P^{(\alpha)}$ where, in the notation of \Cref{prop:OU-intertwining}, $\Lambda_\alpha = \Lambda_{Q \to Q^{(\alpha)}}$, $\Lambda_\delta = \Lambda_{Q^{(\delta)} \to Q}$, and $\Lambda_{\delta,\alpha} = \Lambda_{Q^{(\delta)} \to Q^{(\alpha)}}$.
Hence,
\begin{equation*}
P \inter{\Lambda_\alpha} P^{(\alpha)} \inter{\Lambda_{\delta,\alpha}^* \Lambda_\delta} P \quad \text{and} \quad P^{(\alpha)} \inter{\Lambda_{\delta,\alpha}^* \Lambda_\delta \Lambda_\alpha} P^{(\alpha)} . 
\end{equation*}
\item \label{item-3:prop:OU-quasi-affinity} For $x, \xi \in \R^d$
\begin{equation*}
\Lambda_{\delta,\alpha}^* \Lambda_\delta \Lambda_\alpha e_{i\xi}(x) = e^{-\inn{D_{(\alpha^2-\delta^2)/\alpha)} \xi}{\xi}/2} e_{i\xi}(D_{\frac{\delta}{\alpha}}x).
\end{equation*}
Consequently, with $\bm{t} = b_{\mathrm{min}}^{-1} \log \frac{q_{\infty,\mathrm{max}}}{q_{\infty,\mathrm{min}}}$,
\begin{equation*}
\Lambda_{\delta,\alpha}^* \Lambda_\delta \Lambda_\alpha = P_{\bm{t}}^{(\alpha)}.
\end{equation*}
\end{enumerate}
\end{proposition}

\begin{proof}
Writing $I_d$ for the $d$-dimensional identity matrix we recall that, for the L\"owner ordering of symmetric positive-definite matrices, $q_{\infty,\mathrm{max}} I_d \succ Q_\infty \succ q_{\infty,\mathrm{min}} I_d$. By definition of $\alpha$ it follows that the smallest eigenvalue of the diagonal matrix $D_\alpha$ is $q_{\infty,\mathrm{max}}$, from which we conclude that $D_\alpha \succ q_{\infty,\mathrm{max}}I_d \succ Q_\infty$. Next we recall that $\int_0^\infty e^{-s2B} ds = (2B)^{-1}$. Since $B$, $D_\alpha$, and $e^{-sB}$, for any $s \geq 0$, are diagonal matrices it follows that they commute. Setting $Q^{(\alpha)} = D_\alpha + 2B = D_{\alpha+2b} \succ 0$ we get that
\begin{equation*}
\int_0^\infty	 e^{-sB} Q^{(\alpha)} e^{-sB} ds = D_\alpha  (2B) \int_0^\infty e^{-s2B} ds = D_\alpha.
\end{equation*}
Similarly, setting $Q^{(\delta)} = q_{\infty,\mathrm{min}}2B$ we get that $\int_0^\infty e^{-sB} Q^{(\delta)} e^{-sB} ds = D_\delta$, which proves the first claim. The intertwinings $P \inter{\Lambda_\alpha} P^{(\alpha)}$, $P^\delta \inter{\Lambda_\delta} P$, and $P^\delta \inter{\Lambda_{\delta,\alpha}} P^{(\alpha)}$ then follow from \Cref{prop:OU-intertwining} and the fact that $D_\alpha \succ Q_\infty \succ D_\delta$. The self-adjointness of $P^{(\alpha)}$ and $P^{(\delta)}$ is equivalent to the commutation identities $Q^{(\alpha)}B = BQ^{(\alpha)}$ and $Q^{(\delta)}B = BQ^{(\delta)}$, respectively, see e.g.~\cite[Proposition 9.3.10]{lorenzi:2007}. Taking the adjoint of the identity $P^{(\delta)} \inter{\Lambda_{\alpha,\delta}} P^{(\alpha)}$ and using the self-adjointness of $P^{(\alpha)}$ and $P^{(\delta)}$ then yields $P^{(\alpha)} \inter{\Lambda_{\alpha,\delta}^*} P^{(\delta)}$, which combined with the aforementioned intertwinings finishes the proof of the second claim. For the last claim we get, from the definition of $\Lambda_\alpha$ and $\Lambda_\delta$ that, for all $x,\xi \in \R^d$,
\begin{equation*}
\Lambda_{\delta}\Lambda_\alpha e_{i\xi}(x) = e^{-\inn{(D_\alpha-Q_\infty) \xi}{\xi}/2}  \Lambda_\delta e_{i\xi}(x) e^{-\inn{(D_\alpha-Q_\infty) \xi}{\xi}/2} e^{-\inn{(Q_\infty-D_\delta) \xi}{\xi}/2} e_{i\xi}(x) = e^{-\inn{D_{\alpha-\delta} \xi}{\xi}/2} e_{i\xi}(x).
\end{equation*}
Next, we shall characterize the adjoint operator $\Lambda_{\delta,\alpha}^*$. To this end we note that, since $D_\alpha \succ D_\delta \succ 0$, the invariant measures of $P^{(\alpha)}$ and $P^{(\delta)}$ admit Gaussian densities, which we denote by $\rho_\infty^{(\alpha)}$ and $\rho_\infty^{(\delta)}$, respectively. Let us formally define, for $f \in \Leb^2(\rho_\infty^{(\delta)})$, the operator $\Lambda_{\delta,\alpha}^*:\Leb^2(\rho_\infty^{(\delta)}) \to \Leb^2(\rho_\infty^{(\alpha)})$ by
\begin{equation}
\label{eq:Lambda-adjoint}
\Lambda_{\delta,\alpha}^*f(x) = \frac{1}{\rho_\infty^{(\alpha)}(x)} (f  \rho_\infty^{(\delta)}) \ast \rho_{\delta,\alpha}(x), \quad x \in\R^d,
\end{equation}
where $\rho_{\delta,\alpha}$ is the Gaussian density satisfying $\Fo_{\rho_{\delta,\alpha}}(\xi) = e^{-\inn{(D_\alpha - D_\delta)\xi}{\xi}/2} = e^{-\inn{D_{\alpha - \delta}\xi}{\xi}/2}$, which is well-defined due to $D_\alpha - D_{\delta} = D_{\alpha-\delta} \succ 0$. Then, for non-negative functions $f \in \Leb^2(\rho_\infty^{(\delta)})$ and $g \in \Leb^2(\rho_\infty^{(\alpha)})$,
\begin{align*}
\inn{\Lambda_{\delta,\alpha}^* f}{g}_{\Leb^2(\rho_\infty^{(\alpha)})} = \inn{(f \rho_\infty^{(\delta)}) \ast \rho_{\delta,\alpha}}{g}_{\Leb^2(\R^d)} &= \int_{\R^d} \left(\int_{\R^d} f(y)\rho_\infty^{(\delta)}(y)  \rho_{\delta,\alpha}(x-y) dy \right) g(x) dx \\
&= \int_{\R^d} \left(\int_{\R^d} g(x) \rho_{\delta,\alpha}(y-x) dx \right) f(y)\rho_\infty^{(\delta)}(y) dy = \inn{f}{\Lambda_{\delta,\alpha} g}_{\Leb^2(\rho_\infty^{(\delta)})}
\end{align*}
where we used Fubini's theorem and the symmetry of the density $\rho_{\delta,\alpha}$. By decomposing any $f \in \Leb^2(\rho_\infty^{(\delta)})$ and $g \in \Leb^2(\rho_\infty^{(\alpha)})$ into the difference of non-negative functions it follows that the above holds for all $f \in \Leb^2(\rho_\infty^{(\delta)})$ and $g \in \Leb^2(\rho_\infty^{(\alpha)})$, so that indeed $\Lambda_{\delta,\alpha}^*$ is the $\Leb^2(\rho_\infty^{(\delta)})$ adjoint of the operator $\Lambda_{\delta,\alpha}$. By substituting the expression for the densities in \eqref{eq:Lambda-adjoint} we find that
\begin{align*}
\Lambda_{\delta,\alpha}^* e_{i\xi}(x) &= (2\pi)^{-d/2} \sqrt{\frac{\det D_\alpha }{\det D_\delta \det D_{\alpha-\delta}}} e^{\inn{D_\alpha^{-1} x}{x}/2} \int_{\R^d} e^{i \inn{\xi}{y}} e^{-\inn{D_\delta^{-1}y}{y}/2} e^{-\inn{(D_{\alpha - \delta})^{-1}(x-y)}{(x-y)}/2} dy  \\
&= \frac{(2\pi)^{-d/2}}{(\det D_{\frac{\delta(\alpha-\delta)}{\alpha}})^{1/2}} e^{\inn{(D_{\frac{1}{\alpha}} - (D_{\alpha-\delta})^{-1}) x}{x}/2} \int_{\R^d} e^{\inn{(D_{\alpha - \delta})^{-1}x+i\xi}{y}} e^{-\inn{(D_\delta^{-1}+(D_{\alpha-\delta})^{-1})y}{y}/2} dy \\
&= \frac{(2\pi)^{-d/2}}{(\det D_{\frac{\delta(\alpha-\delta)}{\alpha}})^{1/2}} e^{-\inn{D_{\frac{\delta}{\alpha(\alpha-\delta)}} x}{x}/2} \int_{\R^d} e^{\inn{D_{\frac{1}{\alpha - \delta}}x+i\xi}{y}} e^{-\inn{D_{\frac{\delta(\alpha-\delta)}{\alpha}}^{-1}y}{y}/2} dy  \nonumber \\
&= e^{-\inn{D_{\frac{\delta}{\alpha(\alpha-\delta)}} x}{x}/2}  e^{\inn{D_{\frac{\delta(\alpha-\delta)}{\alpha}} (D_{\frac{1}{\alpha - \delta}}x+i\xi)}{D_{\frac{1}{\alpha - \delta}}x+i\xi}/2} \\
&= e^{-\inn{D_{\frac{\delta}{\alpha(\alpha-\delta)}} x}{x}/2} e^{\inn{D_{\frac{\delta}{\alpha(\alpha-\delta)}} x}{x}/2} e^{i\inn{\xi}{D_{\frac{\delta}{\alpha}} x}} e^{-\inn{D_{\frac{\delta(\alpha-\delta)}{\alpha}} \xi}{\xi}/2} = e^{i\inn{\xi}{D_{\frac{\delta}{\alpha}} x}} e^{-\inn{D_{\frac{\delta(\alpha-\delta)}{\alpha}} \xi}{\xi}/2},
\end{align*}
where in the second equality we expanded the quadratic form, and we repeatedly used some standard properties of diagonal matrices. Putting things together, we deduce that
\begin{equation*}
\Lambda_{\delta,\alpha}^* \Lambda_{\delta}\Lambda_\alpha e_{i\xi}(x) = e^{i\inn{\xi}{D_{\frac{\delta}{\alpha}} x}} e^{-\inn{(D_{\alpha-\delta}+D_{\frac{\delta(\alpha-\delta)}{\alpha}})\xi}{\xi}/2}   = e^{i\inn{\xi}{D_{\frac{\delta}{\alpha}} x}} e^{-\inn{D_{(\alpha^2-\delta^2)/\alpha}\xi}{\xi}/2}.
\end{equation*}
Next, we note that
\begin{equation*}
P_{\bm{t}}^{(\alpha)} e_{i\xi}(x) = e^{i\inn{\xi}{e^{-\bm{t}B}x}} e^{-\inn{ Q_{\bm{t}} ^\alpha \xi}{\xi}/2}.
\end{equation*}
Since $B = D_b$ we get, by definition of $\alpha$, $\delta$ and $\bm{t}$, that the identity $e^{-\bm{t}D_b} = D_{\frac{\delta}{\alpha}}$ is satisfied. Using the identity $Q_{\bm{t}}^{(\alpha)} = Q_\infty^{(\alpha)} - e^{-\bm{t}D_b} Q_\infty^{(\alpha)} e^{-\bm{t}D_b} = D_\alpha - e^{-\bm{t}D_b} D_\alpha e^{-\bm{t}D_b} $ and substituting $e^{-\bm{t}D_b} = D_{\frac{\delta}{\alpha}}$ we find that
\begin{equation*}
Q_{\bm{t}}^{(\alpha)} = D_\alpha - D_{\frac{\delta}{\alpha}} D_\alpha D_{\frac{\delta}{\alpha}} = D_\alpha - D_{\delta^2/\alpha} = D_{(\alpha^2-\delta^2)/\alpha}.
\end{equation*}
This gives that $\Lambda_{\delta,\alpha}^* \Lambda_{\delta}\Lambda_\alpha e_{i\xi}(x) = P_{\bm{t}}^{(\alpha)} e_{i\xi}(x)$ and, as the Fourier transform uniquely characterizes probability measures, the proof is complete.
\end{proof}

We are now able to give the proof of \Cref{thm:OU-convergence}.

\begin{proof}[Proof of \Cref{thm:OU-convergence}]
Since $B$ is diagonalizable with similarity matrix $V$ we have that $VBV^{-1} = D_{b}$, where $b \in \R^d$ is the vector of eigenvalues of $B$ with $b_i > 0$ for all $i$. Under this change of coordinates, $(Q,B)$ gets mapped to $(VQV^*,D_b)$ and a simple calculation shows that $Q_\infty$ then gets mapped to $VQ_\infty V^*$. The change of coordinates map ${\Phi_V} f(x) = f(V^{-1}x)$ is a unitary operator from $\Leb^2(\rho_\infty)$ to $\Leb^2(\rho_\infty^{\Phi_V})$, where $\rho_\infty^{\Phi_V}$ denotes the image density of $\rho_\infty$ under ${\Phi_V}$, i.e.~for $x \in \R^d$, $\rho_\infty^{\Phi_V}(x) = \frac{1}{|\det V|}\rho_\infty({\Phi_V}(x))$. Hence if we prove the desired result for the Ornstein-Uhlenbeck semigroup $\overline{P}$ associated to $(VQV^*,D_b)$ then, since $P_t = {\Phi_V}^{-1} \overline{P}_t {\Phi_V}$ we get, by \Cref{prop:convergence-similarity} and the unitarity of $\Phi_V$, that the claims hold for the Ornstein-Uhlenbeck semigroup $P$ associated to $(Q,B)$.  Thus, we suppose that $B = D_b$ with $b_i > 0$ for all $i$. We aim to invoke \Cref{thm:two-sided-intertwining} and to this end, since $B$ is diagonal, \Cref{prop:OU-quasi-affinity}\ref{item-2:prop:OU-quasi-affinity} furnishes the intertwinings $P \inter{\Lambda^\alpha} P^{(\alpha)} \inter{\Lambda_{\delta,\alpha}^* \Lambda_\delta} P$, where $P^{(\alpha)}$ is the non-degenerate Ornstein-Uhlenbeck semigroup associated to $(Q^{(\alpha)},B)$, see the notation therein. From \Cref{prop:OU-intertwining} we have that $\Lambda_\alpha$, $\Lambda_\delta$, and $\Lambda_{\delta,\alpha}$ are quasi-affinities, and hence $\Lambda_{\delta,\alpha}^* \Lambda_\delta$ is also a quasi-affinity which proves that the intertwinings are proper. Next, \Cref{prop:OU-quasi-affinity}\ref{item-3:prop:OU-quasi-affinity} gives that the function $m:\sigma(\A^{(\alpha)}) \to \C$, in the notation of \Cref{thm:two-sided-intertwining} and where $-\A^{(\alpha)}$ is the generator of $P^{(\alpha)}$, is given by
\begin{equation*}
m(\gamma) = e^{- \gamma \bm{t}},
\end{equation*}
with $\bm{t} = \frac{1}{b_\mathrm{min}} \log \frac{q_{\infty,\mathrm{max}}}{q_{\infty,\mathrm{min}}} = \frac{1}{b_\mathrm{min}} \log \kappa(Q_\infty)$. However, from \cite[Theorem 3.4]{metafune:2002} we get that $\bm{\gamma}_1 = b_\mathrm{min}$, and thus $\bm{t} = \frac{1}{\bm{\gamma}_1} \log \kappa(Q_\infty)$ as claimed. Now, for any $t > \bm{t}$,
\begin{equation*}
\gamma \mapsto \frac{e^{-\gamma t}}{m(\gamma)} = e^{-\gamma(t-\bm{t})} \in \Leb^\infty(\sigma(\A^{(\alpha)})),
\end{equation*}
and plainly
\begin{equation*}
\sup_{\Re(\gamma) \geq \bm{\gamma}_1} e^{-\gamma(t-\bm{t})} = e^{-\bm{\gamma}_1(t-\bm{t})}.
\end{equation*}
Next, \Cref{prop:OU-intertwining} gives that $\norm{\Lambda_\alpha}_{\Leb^2(\rho_\infty^{(\alpha)}) \to \Leb^2(\rho_\infty)} = 1$ and $\norm{\Lambda_{\delta,\alpha}^* \Lambda_\delta}_{\Leb^2(\rho_\infty) \to \Leb^2(\rho_\infty^{(\alpha)})} \leq 1$. To deduce that $\norm{\Lambda_{\delta,\alpha}^* \Lambda_\delta}_{\Leb^2(\rho_\infty) \to \Leb^2(\rho_\infty^{(\alpha)})} = 1$ it suffices to observe that $\Lambda_{\delta,\alpha}^* \bm{1} = \bm{1}$, which follows from \Cref{eq:Lambda-adjoint} and the identity $\rho_\infty^{(\delta)} \ast \rho_{\delta,\alpha} = \rho_\infty^{(\alpha)}$, with the notation therein. Consequently, invoking \Cref{thm:two-sided-intertwining} we deduce that, for any $f \in \Leb^2(\rho_\infty)$ and $t > \bm{t}$,
\begin{equation*}
\norm{P_t f - P_\infty f}_{\Leb^2(\rho_\infty)} \leq e^{- \bm{\gamma_1}(t-\bm{t})}\norm{f - P_\infty f}_{\Leb^2(\rho_\infty)} =  \kappa(Q_\infty) e^{- \bm{\gamma}_1 t}\norm{f - P_\infty f}_{\Leb^2(\rho_\infty)}.
\end{equation*}
However, for $0 \leq t \leq \bm{t}$ we have that $\kappa(Q_\infty) e^{- \bm{\gamma}_1 t} \geq 1$ and thus, by the contractivity of $P$ on $\Leb^2(\rho_\infty)$, it follows that the hypocoercive estimate holds for all $t \geq 0$. Finally, as remarked before the theorem, we have that $P_\infty f(x) = \int_{\R^d} f(x) \rho_\infty(x)dx$.
\end{proof}

\subsection{Proof of \Cref{thm:Jacobi-convergence}} \label{subsec:proof-thm:Jacobi-convergence}

In this proof we use standard properties of tensor products of semigroups and generators, see for instance~\cite[Section 1.15.3]{bakry:2014}. Let us write $P^{(i)}$ for the one-dimensional factors of the product semigroup $P$. By \cite[Proposition 3.6]{cheridito:2019} we get that, for each $i = 1,\ldots,d$, there exists a one-dimensional classical Jacobi semigroup $\tilde{P}^{(\m,i)} = (e^{-t\tilde{\A}_{\m,i}})_{t \geq 0}$ on $\Leb^2(\beta_\m)$ such that $P^{(i)} \inter{\Lambda_{\m,i}} \tilde{P}^{(\m,i)} \inter{\tilde{\Lambda}_{\m,i}} P^{(i)}$, where $\Lambda_{\m,i} \in \B(\Leb^2(\beta_{\m}),\Leb^2(\beta_i))$ and $\tilde{\Lambda}_{\m,i} \in \B(\Leb^2(\beta_i),\Leb^2(\beta_\m))$ are quasi-affinities with operator norm 1, such that
\begin{equation*}
\tilde{\Lambda}_{\m,i} \Lambda_{\m,i} = F_\m(\tilde{\A}_{\m,i}),
\end{equation*}
see Proposition 3.5 and Lemma 3.10, respectively, of the same paper, and where we note that the quantity $\bm{d}$, in the notation therein, may be taken to be 1. Since the parameter $\m$ is common to all factors of the product semigroup we get, by tensorization, the intertwinings $P \inter{\Lambda_\m} \tilde{P}^{(\m)} \inter{\tilde{\Lambda}_\m} P$, where $\Lambda_\m$ acts on $f \in \Leb^2(\bpsi)$ via $\Lambda_\m f(x) = \Lambda_{\m,1} f_1(x_1) \cdots \Lambda_{\m,d} f_d(x_d)$, and similarly for $\tilde{\Lambda}_\m$, and plainly both $\Lambda_\m$ and $\tilde{\Lambda}_\m$ are also quasi-affinities, hence proper linear operators. Next, the fact that $F_\m \in \Leb^\infty(\sigma(\tilde{\A}_\m))$ follows from $F_\m$ being the Laplace transform of a probability measure on $[0,\infty)$, see Section 3.7 in the same paper, so that $|F_\m(\gamma)| < \infty$ for any $\Re(\gamma) \geq 0$. Recall that $\sigma(\tilde{\A}_\m) = \{ \gamma_n = n(n-1)+\bm{\gamma}_1 n; n \in \N \}$ and thus, for $\bm{\gamma}_1 t \geq \log \frac{\bm{\gamma}_1(\m+1)}{2(\bm{\gamma}_1-\m+1)}$,
\begin{equation*}
\sup_{\gamma \geq \bm{\gamma}_1} \frac{e^{-\gamma t}}{F_\m(\gamma)} = \sup_{n \geq 1} \frac{e^{-\gamma_n t}}{F_\m(\gamma_n)} \leq \sup_{n \geq 1} \frac{e^{- n \bm{\gamma_1} t}}{F_\m(\gamma_n)} \leq \m \frac{(\bm{\gamma}_1-1)}{(\bm{\gamma}_1-\m)} e^{-\bm{\gamma}_1 t},
\end{equation*}
see Equation (3.41) in \cite{cheridito:2019}. Hence, by \Cref{thm:two-sided-intertwining}, we deduce the convergence to equilibrium estimate
\begin{equation}
\label{eq:gJac-hypo}
\norm{P_t f - P_\infty f}_{\Leb^2(\beta)} \leq \m \frac{\bm{(\gamma}_1-1)}{(\bm{\gamma}_1-\m)} e^{-\bm{\gamma}_1 t} \norm{f-P_\infty f}_{\Leb^2(\beta)},
\end{equation}
which is valid for all $f \in \Leb^2(\beta)$ and $t \geq \frac{1}{\bm{\gamma}_1}\log \frac{\bm{\gamma}_1(\m+1)}{2(\bm{\gamma}_1-\m+1)}$. However, for any $0 \leq t < \frac{1}{\bm{\gamma}_1}\log \frac{\bm{\gamma}_1(\m+1)}{2(\bm{\gamma}_1-\m+1)}$ it is straightforward to check that the constant in front of the exponential in \eqref{eq:gJac-hypo} is strictly greater than 1 so that, by the contractivity of $P$, the estimate \eqref{eq:gJac-hypo} holds for all $f \in \Leb^2(\beta)$ and $t \geq 0$. Recalling that $P_\infty f = \int_{[0,1]^d} f(x)\beta(x)dx$ we complete the proof of the claimed hypocoercive estimate.

\bibliographystyle{abbrv}

\end{document}